\documentclass[leqno,10pt]{article}\textwidth 175mm\textheight 235mm
\usepackage{amsfonts} 
\usepackage{amssymb} 
\usepackage{graphicx}
\usepackage{amsmath}
\usepackage{amsthm} 
\usepackage{mathrsfs}
\usepackage{tikz} 
\usepackage{rotating,tabularx}
\usepackage{cancel}\usepackage{todonotes}
\usetikzlibrary{matrix} \usetikzlibrary{arrows,shapes}
\usepackage{cite}\usepackage{hyperref}
\hypersetup{colorlinks=true,urlcolor=blue}
\expandafter\let\expandafter\oldproof\csname\string\proof\endcsname
\let\oldendproof\endproof \renewenvironment{proof}[1][\proofname]{
\oldproof[\ttfamily\scshape \bf #1.] }{\oldendproof}
\def\ve{\varepsilon}  \def\tilde{\widetilde}
\def\emp{\emptyset}  
\def\dom{{\rm dom}\,} 
 \def\min{\mbox{\rm minimize}}

   \def\B{\mathbb
B}   
  \def\disp{\displaystyle}
\def\tto{\rightrightarrows} 
 \def\Bar{\overline} \def\ra{\rangle}
\def\la{\langle} \def\ve{\varepsilon} \def\epsilon{\varepsilon}
\def\ox{\bar{x}} 
\def\oy{\bar{y}} 
\def\oz{\bar{z}}
 
\def\ov{\bar{v}}

\def\kk{\kappa}

\def\gph{\mbox{\rm gph}\,} 
 \def\dom{\mbox{\rm dom}\,}
\def\ker{\mbox{\rm ker}\,}

\def\rank{\mbox{\rm rank}\,}
  \def\O{\Omega}
\def\ph{\varphi} \def\emp{\emptyset}
\def\oR{\Bar{\R}} \def\lm{\lambda}\def\gg{\gamma} \def\dd{\delta}
\def\al{\alpha}  \def\N{{\rm I\!N}}
\def\R{{\rm I\!R}} \def\th{\theta} \def\vt{\vartheta}
 
\def\Limsup{\mathop{{\rm Lim}\,{\rm sup}}}

\def\sce{\setcounter{equation}{0}}
\begin{document}
\newtheorem{Assumption}{Assumption}
\newtheorem{Theorem}{Theorem}[section]
\newtheorem{Proposition}[Theorem]{Proposition}
\newtheorem{Remark}[Theorem]{Remark}
\newtheorem{Lemma}[Theorem]{Lemma}
\newtheorem{Corollary}[Theorem]{Corollary}
\newtheorem{Definition}[Theorem]{Definition}
\newtheorem{Example}[Theorem]{Example}
\newtheorem{Algorithm}[Theorem]{Algorithm}
\newtheorem{Problem}[Theorem]{Problem}
\renewcommand{\theequation}{{\thesection}.\arabic{equation}}
\renewcommand{\thefootnote}{\fnsymbol{footnote}} \begin{center}
{\bf\Large Variational Convexity of Functions and\\
Variational Sufficiency in Optimization}\\[1ex]
PHAM DUY KHANH\footnote{Department of Mathematics, HCMC University
of Education, Ho Chi Minh City, Vietnam. E-mail:
pdkhanh182@gmail.com.}, BORIS S.
MORDUKHOVICH\footnote{Department of Mathematics, Wayne State
University, Detroit, Michigan, USA. E-mail: aa1086wayne.edu. Research of this author was partly supported by the US National Science Foundation under grants DMS-1808978 and DMS-2204519, by the Australian Research Council under Discovery Project DP-190100555, and by Project 111 of China under grant D21024.}
and VO THANH PHAT \footnote{Department of Mathematics, HCMC
University of Education, Ho Chi Minh City, Vietnam and Department of
Mathematics, Wayne State University, Detroit, Michigan, USA.
E-mails: phatvt@hcmue.edu.vn; phatvt@wayne.edu. The research of this author was partly supported by the US National Science Foundation under grants DMS-1808978 and DMS-2204519.}
\end{center}
\small
{\bf Abstract}. The paper is devoted to the study, characterizations, and applications of variational convexity of functions, the property that has been recently introduced by Rockafellar together with its strong counterpart. First we show that these variational properties of an extended-real-valued function are equivalent to, respectively, the conventional (local) convexity and strong convexity of its Moreau envelope. Then we derive new characterizations of both variational convexity and variational strong convexity of general functions via their second-order subdifferentials (generalized Hessians), which are coderivatives of subgradient mappings. We also study relationships of these notions with local minimizers and tilt-stable local minimizers. The obtained results are used for characterizing related notions of variational and strong variational sufficiency in composite optimization with applications to nonlinear programming.\\[0.05ex] 
\vspace*{-0.15in}
		
\normalsize
\section{Introduction}\label{sec:intro}\sce\vspace*{-0.05in}
\setcounter{equation}{0}

The properties of {\em variational convexity} and {\em variational strong convexity} of extended-real-valued lower semicontinuous (l.s.c.) functions $\ph\colon\R^n\to\oR:=(-\infty,\infty]$ have been recently introduced by Rockafellar \cite{r19}. It has been well recognized and utilized in convex and variational analysis that the subgradient mapping $\partial\ph$ associated with  an l.s.c.\ function $\ph$ is maximal monotone {\em if and only if} the function is convex on the entire space. This result is fundamental in convex analysis and its numerous applications, particularly those to optimization-related problems; see, e.g., the books \cite{Bauschke,Mordukhovich06,mornam,Rockafellar70,Rockafellar98} for more details and references therein. Loosely speaking, variational convexity corresponds to the {\em graphical localization} of the maximal monotonicity property in terms of the limiting subdifferential around the point in question. Variational strong convexity can be defined in this scheme with replacing monotonicity of the subgradient mapping by strong monotonicity (with some modulus). The latter property is closely related to the notion of {\em tilt stability} of local minimizers introduced earlier in Poliquin and Rockafellar \cite{Poli}.

Both variational convexity and variational strong convexity properties of a function are more subtle than the conventional notions of {\em local convexity} and {\em strong local convexity} of the function relative to some neighborhood of the reference point. It is easy to illustrate by simple one-dimensional examples that the corresponding {\em variational} properties are satisfied without any local convexity. The importance of variational convexity and its strong counterpart for the study of nonconvex optimization problems has been demonstrated in the recent papers by Rockafellar \cite{r19,roc,r22}, where the reader can find, in particular, applications to the proximal point and augmented Lagrangian methods in rather general frameworks. 

Our first {\em principal} result in this paper reveals that the {\em variational convexity} of an l.s.c.\ prox-bounded function is {\em equivalent} to the usual {\em local convexity} of the {\em Moreau envelope} together with the {\em prox-regularity} of the function in question. We prove in this way that minimizing a {\em nonsmooth variationally convex} function can be reduced to the minimization of its {\em smooth and convex} Moreau envelope. A parallel characterization is established for {\em variational strong convexity} of extended-real-valued functions. 

Since both Moreau envelope and prox-regularity notions are well understood and employed in variational analysis and optimization, the obtained characterizations open the door for further developments and applications. Note that the reduction of minimizing prox-regular functions to the minimization of their smooth Moreau envelopes have been recently exploited in \cite{BorisKhanhPhat,kmptjogo,kmptmp,BorisEbrahim} for the design and justification of generalized {\em Newton-type algorithms} in nonsmooth optimization. The additional {\em convexity} of Moreau envelopes in the case of variationally convex cost functions creates new numerical perspectives for this approach by applying powerful theoretical and algorithmic tools of convex analysis and optimization. 

The aforementioned principal results and machinery of second-order variational analysis allow us to derive coderivative-based {\em second-order characterizations} of variational convexity and variational strong convexity with prescribed moduli for general classes of l.s.c.\ functions. These characterizations are obtained in both {\em neighborhood} and {\em pointbased} forms via the {\em second-order subdifferentials} introduced by Mordukhovich \cite{m92}, which have been broadly developed in variational theory and applications, including more recent ones to numerical algorithms of nonsmooth optimization; see below. In this way we also shed new light on the study of {\em tilt-stable minimizers}.

Related topics addressed in this paper concern problems of {\em composite optimization} written in the form
\begin{equation}\label{cop}
\min\quad \ph(x)+\psi\big(g(x)\big),\quad x\in\R^n,
\end{equation}
where $\ph\colon\R^n\to\R$ and $g\colon\R^n\to\R^m$ are ${\cal C}^2$-smooth, while $\psi\colon\R^m\to\oR$ is merely lower semicontinuous. Since the function $\psi$ is generally extended-real-valued, format \eqref{cop} includes problems of constrained optimization with the domain constraints $g(x)\in\dom\psi$. Our major attention here is paid to Rockafellar's recent notions of {\em variational and strong variational sufficiency} for local optimality in \eqref{cop} that are proved to be important for developing both theoretical and computational aspects of optimization. Based on the second-order subdifferential characterizations of variational convexity and its strong counterpart, we derive complete {\em characterizations of variational and strong variational sufficiency} for large classes of composite optimization problems. The obtained results are specified for problems of {\em nonlinear programming}, where they are expressed entirely via the program data due to the explicit computations of second-order subdifferentials.  

The rest of the paper is organized as follows. Section~\ref{sec:prelim} overviews those notions of variational analysis and generalized differentiation, which are broadly used in the formulations and proofs of the main results obtained below. In Sections~\ref{sec:Moreau} and \ref{sec:Moreau1}, we recall and discuss the notions of variational convexity and strong variational convexity, respectively, and establish the equivalence between these properties of the function in question and local convexity/strong local convexity of its Moreau envelope. Section~\ref{sec:codcharacterizationconvex} is devoted to deriving neighborhood and pointbased characterizations of variational convexity and variational strong convexity of extended-real-valued functions in terms of coderivative-based second-order subdifferentials. Section~\ref{sec:varsuf} addresses second-order subdifferential characterization of variational and strong variational sufficiency for local optimality in problems of composite optimization \eqref{cop}. Applications of these results to nonlinear programs are given in Section~\ref{vs-nlp}. The concluding Section~\ref{sec:conclusion} summarizes the main results of the paper and discusses some important topics of our future research.\vspace*{-0.17in}

\section{Preliminaries and Discussions}\label{sec:prelim}\vspace*{-0.05in}
\setcounter{equation}{0}

In this section, we recall and discuss some basic notions and facts from variational analysis that are largely in what follows; see \cite{Mordukhovich06,Mor18,Rockafellar98} for more details. Let $F$ be a set-valued mapping (multifunction) between Euclidean spaces $\R^n$ and $\R^m$. As usual, the effective domain and the graph of $F$ are given, respectively, by
$$
\dom F:=\big\{x\in\R^n\;\big|\;F(x)\ne\emp\big\}\quad\text{and}\quad\gph F=\big\{(x,y)\in\R^n\times\R^m\;\big|\;y\in F(x)\big\}.
$$
The (Painlev\'e-Kuratowski) {\em outer limit} of $F$ as $x\rightarrow\bar{x}$ is defined as
\begin{equation}\label{OuterLimit}
	\underset{x\rightarrow\bar{x}}{\Limsup}\;F(x):
	=\big\{y\in\R^n\;\big|\;\exists\,x_k\to\bar x,\ y_k\rightarrow y\;\mbox{ with } y_k\in F(x_k),\;k=1,2,\ldots\big\}.
\end{equation}
Considering an extended-real-valued function $\varphi\colon\R^n\rightarrow\overline{\R}$, we always assume that $\varphi$ is proper, i.e., $\dom\ph:=\{x\in\R^n\;|\;\varphi(x)<\infty\}\ne\emp$. The (Fr\'echet) {\em regular subdifferential} of $\varphi$ at $\bar{x}\in\dom\varphi$ is
\begin{equation}\label{FrechetSubdifferential}
	\widehat\partial\varphi(\ox)=\Big\{v\in\R^n\;\Big|\;\liminf\limits_{x\to \overline x}\frac{\varphi(x)-\varphi(\ox)-\langle v,x-\ox\rangle}{\|x-\ox\|}\ge 0\Big\}.
\end{equation}
The (Mordukhovich) {\em limiting/basic/general subdifferential} and {\em singular/horizon subdifferential} of $\varphi$ at $\ox$ are defined, respectively, via the outer limit \eqref{OuterLimit} by
\begin{equation}\label{MordukhovichSubdifferential}
	\partial\varphi(\bar{x}):=\underset{x \overset{\varphi}{\to} \bar{x}}{\Limsup}\; \widehat{\partial}\varphi(x)\;\mbox{ and }\;
	\partial^\infty\varphi(\bar{x}):=\underset{x\overset{\varphi}{\to} \bar{x} \atop \lambda\downarrow 0}{\Limsup}\;\lambda\widehat{\partial}\varphi(x),
\end{equation}
where $x\overset{\varphi}{\to}\bar{x}$ means that $x\rightarrow\bar{x}$ with $\varphi(x)\rightarrow\varphi(\bar{x})$. In the case where $\widehat{\partial}\varphi(\ox)=\partial\varphi(\ox)$, the function $\varphi$ is called {\em lower regular} at $\ox$; see \cite{Mordukhovich06}. This agrees with the subdifferential (or Clarke) regularity of $\ph$ at $\ox$ in the sense of \cite{Rockafellar98} provided that $\ph$ is locally Lipschitzian around $\ox$. Observe that both regular and limiting subdifferentials reduce to the classical gradient $\nabla\ph(\ox)$ for continuously differentiable functions, while the singular subdifferential of an l.s.c. function reduces to $\{0\}$ if and only if $\ph$ is locally Lipschitzian around $\ox$.

Given a set $\Omega\subset\R^n$ with its indicator function $\delta_\Omega(x)$ equal to $0$ for $x\in\Omega$ and to $\infty$ otherwise, the {\em regular} and {\em limiting normal cones} to $\Omega$ at $\bar{x}\in\Omega$ are defined, respectively, via the subdifferentials \eqref{FrechetSubdifferential} and \eqref{MordukhovichSubdifferential} by
\begin{equation}\label{NormalCones}
	\widehat{N}_\Omega(\bar{x}):=\widehat{\partial}\delta_\Omega(\bar{x})
	\quad\text{and}\quad
	N_\Omega(\bar{x}):=\partial\delta_\Omega(\bar{x}).
\end{equation}
The coderivative constructions for $F\colon\R^n\tto\R^m$ at $(\ox,\oy)\in\gph F$ are defined via the normal cones \eqref{NormalCones} to the graph of $F$ at this point. They are the {\em regular coderivative} and the {\em limiting coderivative} of $F$ at $(\ox,\oy)$ given by
\begin{equation}\label{reg-cod} 
	\widehat{D}^*F(\ox,\oy)(v):=\big\{u\in\R^n\;\big|\;(u,-v)\in \widehat{N}_{{\rm
			gph}\,F}(\ox,\oy)\big\},\quad v\in\R^m, 
\end{equation}
\begin{equation}\label{lim-cod}
	D^*F(\ox,\oy)(v):=\big\{u\in\R^n\;\big|\;(u,-v)\in N_{{\rm
			gph}\,F}(\ox,\oy)\big\},\quad v\in\R^m, 
\end{equation} 
respectively. In the case where $F(\bar{x})$ is the singleton $\{\bar{y}\}$, we omit $\oy$ in the notation of \eqref{reg-cod} and \eqref{lim-cod}. Note that if
$F\colon\R^n\to\R^m$ is ${\cal C}^1$-smooth around $\ox$, then
\begin{equation*}
	\widehat{D}^*F(\bar{x})(u)=D^*F(\bar{x})(u)=\big\{\nabla F(\bar{x})^*u\big\},\quad u\in\R^n,
\end{equation*}
where $\nabla F(\bar{x})^*$ is the adjoint/transpose
matrix of the Jacobian $\nabla F(\bar{x})$.

\begin{Definition}[\bf second-order subdifferentials]\label{2nd} \rm 
	Let $\varphi:\R^n\to \overline{\R}$ and $\ox\in\dom\ph$.
	
	{\bf(i)} For any $\bar{y}\in{\partial}\varphi(\bar{x})$, the mapping ${\partial}^2 \varphi(\bar{x},\bar{y}):\R^n \rightrightarrows\R^n$ with the values
	\begin{equation}\label{limitsec}
		{\partial}^2\varphi(\bar{x},\bar{y})(u):=({D}^* {\partial}\varphi)(\bar{x},\bar{y})(u),\quad u\in\R^n,
	\end{equation} 
	is said to be the \textit{basic/limiting second-order subdifferential} of $\varphi$ at $\bar{x}$ relative to $\bar{y}$.

	{\bf(ii)} For any $\bar{y}\in\partial\varphi(\bar{x})$, the mapping $\breve{\partial}^2\varphi(\bar{x},\bar{y}):\R^n \rightrightarrows \R^n$ with the values
	\begin{equation}\label{seccombine} 
		\breve{\partial}^2\varphi(\bar{x},\bar{y})(u):=(\widehat{D}^*\partial\varphi)(\bar{x},\bar{y})(u),\quad u\in\R^n,
	\end{equation} 
	is said to be the \textit{combined second-order subdifferential} of $\varphi$ at $\bar{x}$ relative to $\bar{y}$. 
\end{Definition}

Clearly, we have the following inclusion
\begin{equation}\label{2ndinclusioncombinelimit}
	\breve{\partial}^2\varphi(x,y)(w)\subset\partial^2\varphi(x,y)(w)\;\text{ for all }\; (x,y)\in\gph\partial\varphi,\; w\in\R^n.
\end{equation}
We omit $\bar{y}=\nabla\varphi(\bar{x})$ in the above second-order subdifferential notation if $\varphi$ is $\mathcal{C}^1$-smooth around $\bar{x}$. If $\varphi$ is $\mathcal{C}^2$-smooth around $\bar{x}$, then we get, via the  {symmetric Hessian matrix that}
\begin{equation*}
	\breve{\partial}^2\varphi(\bar{x})(u)=\partial^2\varphi(\bar{x})(u)=\big\{\nabla^2 \varphi(\bar{x})u\big\}\;\mbox{ for all }\;u\in\R^n.
\end{equation*}
This justifies the names of {\em generalized Hessians} for the second-order constructions from Definition~\ref{2nd}.\vspace*{0.03in}

It has been recognized in variational analysis and applications to optimization and related topics that the basic second-order subdifferential \eqref{limitsec} enjoys well-developed {\em calculus rules} in both finite and infinite dimensions \cite{Mordukhovich06,Mor18,mo,BorisOutrata,mr,msar} and admits efficient {\em computations} for major classes of extended real-valued functions encountered in variational analysis, optimization, machine learning, statistics, stochastic systems, optimal control, etc. as, e.g., in \cite{chhm,dsy,dr,hmn,hos,hr,BorisKhanhPhat,kmptjogo,mr,msar,os,yy}. Involving these calculus and computations, the second-order construction \eqref{limitsec} has been instrumental, among other applications, to provide complete characterizations of the fundamental notions of {\em tilt} and {\em full stability} in optimization, optimal control, and variational systems as, e.g., in \cite{dl,dmn,gm,lpr,MorduNghia13,MorduNghia,MorduNghia1,mnghia-fs,mnr,mor,mr,msar,Poli,qw}, to characterize global and local {\em monotonicity} properties of subgradient mappings \cite{clmn,MorduNghia1}, and to characterize {\em convexity} and {\em generalized convexity} properties of various classes of extended-real-valued functions; see \cite{ChieuChuongYaoYen,ChieuHuy11,clmn,kp18,kp20}. 

Although calculus rules available for the combined second-order subdifferential \eqref{seccombine} are less impressive in comparison with those for \eqref{limitsec}, the second-order construction \eqref{seccombine} is proved to be useful, especially in infinite dimensions, to establish neighborhood characterizations of tilt and full stability properties of variational systems, monotonicity properties of multifunctions, and (generalized) convexity properties of extended-real-valued functions; see, e.g., \cite{ChieuChuongYaoYen,ChieuHuy11,clmn,MorduNghia,MorduNghia1,mnghia-fs,msar} and the references therein. Note that it is often easier to compute \eqref{seccombine} than \eqref{limitsec}, and then to employ the computation of \eqref{seccombine} for the subsequent computation of the more robust construction \eqref{limitsec}. 

We need to  recall yet another notion of generalized second-order derivatives including \textit{second subderivatives} \cite{Rockafellar98} and \textit{quadratic bundles} \cite{roc}. Given $\ph\colon\R^n\to\oR$ with $\ox\in\dom\ph$, define the {\em second subderivative} of $\varphi$ at $\ox$ for $v\in\R^n$ and $w\in\R^n$ by the lower limit of the second-order quotients
\begin{equation*}
	d^2\varphi(\ox,v)(w):=\liminf_{\tau\downarrow 0\atop u\to w} \frac{\varphi(\bar{x}+\tau u)-\varphi(\bar{x})-\tau\langle v,u\rangle}{\frac{1}{2}\tau^2}.
\end{equation*}
Then $\ph$ is said to be {\em twice epi-differentiable} at $\ox$ for $v$ if for every $w\in\R^n$ and every choice of $\tau_k\downarrow 0$, there exists a sequence $w^k\to w$ such that
\begin{equation*}
	\frac{\varphi(\ox+\tau_k w^k)-\varphi(\ox)-\tau_k\langle v,w^k\rangle}{\frac{1}{2}\tau_k^2}\to d^2\varphi(\ox,v)(w)\;\mbox{ as }\;k\to\infty.
\end{equation*}
Twice epi-differentiability has been recognized as an important concept of second-order variational analysis with numerous applications to optimization; see the aforementioned monograph by Rockafellar and Wets and the recent papers \cite{mms,ash-ebr}. A function $\ph:\R^n\to \overline{\R}$  is called a \textit{generalized quadratic form} if $\ph(0)=0$ and the mapping $\partial\ph$ is generalized linear, i.e., $\gph\partial\ph$ is a subspace of $\R^n\times\R^n$. The function $\varphi$ is called \textit{generalized twice differentiable} at $\ox$ for a subgradient $\ov\in \partial\varphi(\ox)$ if it is twice epi-differentiable at $\ox$ for $\ov$ with the second-order subderivative $d^2\varphi(\ox,\ov)$ being a generalized
quadratic form. This allows us to define the \textit{quadratic bundle} of $\varphi$ at $\ox$ for $\ov$ by
\begin{equation}\label{qb}
	\text{\rm quad}\;\varphi(\ox,\ov) := \left[\begin{array}{l}
		\text { the collection of generalized quadratic forms } q \text { for which }\\
		\exists\left(x_k, v_k\right) \rightarrow(\ox, \ov) \text { with } \varphi \text { generalized twice differentiable } \\
		\text { at } x_k \text { for } v_k \text { and such that the generalized quadratic } \\
		\text { forms } q_k=\frac{1}{2} d^2 \varphi\left(x_k,v_k\right) \text { converge epigraphically to } q,
	\end{array}\right.
\end{equation}  

Next we recall some classes of extended-real-valued functions $\ph\colon\R^n\to\oR$ broadly used in the paper. As usual, $\ph$ is {\em convex} on a convex set $\O\subset\R^n$ if
\begin{equation*}
	\varphi\big((1-\lambda)x+\lambda y\big)\le(1-\lambda)\varphi(x)+\lambda\varphi(y)\quad\text{for all }\;x,y\in\Omega,\; \lambda\in[0,1]. 
\end{equation*}
We say that $\ph$ is {\em strongly convex} on $\O$ with modulus $\kk>0$ if its quadratic shift $\varphi-(\kappa/2)\|\cdot\|^2$ is convex on $\O$, i.e., 
\begin{equation}\label{ineqstrongfunction}
	\varphi((1-\lambda)x +\lambda y)\le(1-\lambda)\varphi(x)+\lambda\varphi(y)-\frac{\kappa}{2}\lambda(1-\lambda)\|x-y\|^2,
\end{equation}
for any $x,y\in\Omega$, $\lambda\in[0,1]$. 
It is easy to see that if $\varphi$ is strongly convex on $\Omega$ with modulus $\kappa>0$, then the function $h$ defined by $h(x):=\varphi(ax+b)$ is strongly convex on $\Omega$ with modulus $\kappa a^2$, where $a\ne 0$ and $b\in\R$ are taken from \eqref{ineqstrongfunction}.

An l.s.c.\ function $\ph$ is {\em prox-regular} at $\bar{x}\in\dom\ph$ for $\bar{v}\in\partial\ph(\ox)$ if there exist $\epsilon>0$ and $r\ge 0$ such that 
\begin{equation}\label{prox}
	\varphi(x)\ge\varphi(u)+\langle v,x-u\rangle-\frac{r}{2}\|x-u\|^2\
\end{equation}
for all $x\in\mathbb{B}_\epsilon(\bar{x})$ and $(u,v)\in\gph\partial\varphi\cap (\mathbb{B}_\epsilon(\ox)\times\mathbb{B}_\epsilon(\ov))$ with $\varphi(u)<\varphi(\bar{x})+\epsilon$, where $\B_\ve(a)$ stands for the closed ball centred at $a$ with radius $\ve$. If this holds for all $v\in\partial\varphi(\ox)$, then $\varphi$ is said to be {\em prox-regular at} $\ox$.  

We say that $\varphi$ is {\em subdifferentially continuous} at $\bar{x}$ for $\ov\in \partial\varphi(\bar{x})$ if for any $\epsilon>0$ there exists $\delta>0$ such that $|\varphi(x)-\varphi(\bar{x})|<\epsilon$ whenever $(x,v)\in\gph\partial\varphi\cap (\mathbb{B}_\delta(\ox)\times\mathbb{B}_\delta(\ov))$. When this holds for all $v\in\partial\varphi(\ox)$, the function $\varphi$ is said to be {\em subdifferentially continuous at} $\ox$. It is easy to see that if $\varphi$ is subdifferentially continuous at $\bar{x}$ for $\bar{v}$, then the inequality  ``$\varphi(x)<\varphi(\ox)+\epsilon$" in the definition of prox-regularity can be omitted.  Functions that are both prox-regular and subdifferentially continuous are called {\em continuously prox-regular}. This is a major class of extended-real-valued functions in second-order variational analysis, being a common roof for particular collections of functions important for applications as, e.g., amenable functions, etc.; see \cite[Chapter~13]{Rockafellar98}.\vspace*{0.03in}

The following two well-known constructions play a crucial role in this paper. Given 
an l.s.c.\ function $\varphi:\R^n\to\overline{\R}$ and a parameter value $\lambda>0$, the {\em Moreau envelope} $e_\lambda\varphi$ and {\em proximal mapping} $\textit{\rm Prox}_{\lambda\varphi}$ are defined by
\begin{equation}\label{Moreau} 
	e_\lambda\varphi(x):=\inf\Big\{\varphi(y)+\frac{1}{2\lambda}\|y-x\|^2\;\Big|\;y\in \R^n\Big\},\quad x\in\R^n,
\end{equation}	
\begin{equation}\label{ProxMapping} 
	\operatorname{Prox}_{\lambda\varphi}(x):= \operatorname{argmin}\Big\{\varphi(y)+ \frac{1}{2\lambda}\|y-x\|^2\;\Big|\;y\in\R^n\Big\},\quad x\in\R^n.
\end{equation}	
A function $\varphi$ is said to be {\em prox-bounded} if there exists $\lambda>0$ such that $e_\lambda\varphi(x)>-\infty$ for some $x\in\R^n$.\vspace*{0.03in}

We now present important properties of the Moreau envelope and the proximal mapping that are taken from \cite[Proposition~13.37]{Rockafellar98}. Recall that the {\em $\varphi$-attentive $\epsilon$-localization} of the subgradient mapping $\partial\varphi$ around $(\ox,\ov)$ is the set-valued mapping $T\colon\R^n\tto\R^n$ defined by
\begin{equation}\label{attentive}
	T(x):=\begin{cases}
		\big\{v\in\partial\varphi(x)\;\big|\;\|v-\ov\|<\epsilon\big\} & \text{if}\quad \|x-\ox\|<\epsilon\; \text{ and }\; |\varphi(x)-\varphi(\ox)|<\epsilon,\\
		\emp &\text{otherwise}.
	\end{cases}
\end{equation} 
If $\varphi$ is an l.s.c.\ function, a localization can be taken with just  $\varphi(x)<\varphi(\ox)+\epsilon$ in \eqref{attentive}.  Recall also that a function $\ph$ is of {\em class $C^{1,1}$}  (or ${\cal C}^{1.+})$) around $\ox$ if it is ${\cal C}^1$-smooth and its gradient is Lipschitz continuous around this point.

\begin{Proposition}[\bf Moreau envelopes and proximal mappings for prox-regular functions]\label{C11}  Let $\varphi\colon\R^n\to\oR$ be an l.s.c.\ and prox-bounded function which is prox-regular at $\ox$ for $\ov\in\partial\varphi(\ox)$. Then there exists a $\varphi$-attentive $\epsilon$-localization $T$ of $\partial\varphi$ such that for all sufficiently small numbers $\lambda>0$ there is a convex neighborhood $U_\lambda$ of $\ox+\lambda\ov$ on which the following hold:\\[1ex]
	{\bf(i)} The Moreau envelope $e_\lambda\varphi$ from \eqref{Moreau} is of class ${\cal C}^{1,1}$ on the set $U_\lambda $.\\[1ex]
	{\bf(ii)} The proximal mapping $\text{\rm Prox}_{\lm\ph}$ from \eqref{ProxMapping}  is single-valued, monotone, and Lipschitz continuous on $U_\lm$ satisfying the
	condition ${\rm Prox}_{\lm\ph}(\ox+\lm\ov)=\ox$.\\[1ex]
	{\bf(iii)} The gradient of $e_\lm\ph$ is calculated by:
	\begin{equation}\label{GradEnvelope} 
		\nabla e_\lambda\varphi(x)=\frac{1}{\lambda}\Big(x-\text{\rm Prox}_{\lambda\varphi}(x)\Big)=\big(\lambda I+T^{-1}\big)^{-1}(x)\;\mbox{ for all }\;x\in U_\lambda.
	\end{equation}
	If in addition $\varphi$ is subdifferentially continuous at $\ox$ for $\ov$, then $T$ in \eqref{GradEnvelope} can be replaced by $\partial\varphi$. 
\end{Proposition}

Next we formulate the main properties of extended-real-valued functions studied in this paper, which have been recently introduced and investigated by Rockafellar in \cite{r19,roc,r22}.

\begin{Definition}[\bf variationally convex functions]\label{vr} \rm An l.s.c.\  function $\varphi:\R^n\to\overline{\R}$ is called {\em variationally convex} at $\ox$ for $\ov\in\partial\varphi(\ox)$ if for some convex neighborhood $U\times V$ of $(\ox,\ov)$ there exist an l.s.c.\ convex function $\psi\le\varphi$ on $U$ and a number $\epsilon>0$ such that 
	\begin{equation}\label{varconvex}
		(U_\epsilon\times V)\cap\gph\partial\varphi=(U\times V)\cap\gph\partial\psi\;\text{and}\;\varphi(x)=\psi(x)\;\text{at the common elements}\;(x,v), 
	\end{equation}
	where $U_\epsilon:=\{x\in U\;|\;\varphi(x)<\varphi(\ox)+\epsilon\}$. We say that $\ph$ is {\em variationally strongly convex} at $\ox$ for $\ov$ with modulus $\sigma >0$ if \eqref{varconvex} holds with $\psi$ being strongly convex on $U$ with this modulus.
\end{Definition}

Let us illuminate some remarkable features of Rockafellar's notions from Definition~\ref{vr}.

\begin{Remark}[\bf discussion on variational and strong variational convexity]\label{disvar} {\rm Observe the following:
		
		{\bf(i)} It is easy to see from Definition~\ref{vr} that if $V=\R^n$ in \eqref{varconvex}, then the variational convexity (variational strong convexity) reduces to the local convexity (local strong convexity) of $\varphi$ around $\ox$, but not necessarily otherwise; cf. \cite[Examples~6 and 7]{r19}. 
		
		{\bf(ii)} In Definition~\ref{vr} of variational convexity, the function $\psi$ is {\em locally convex} (i.e., convex on $U$). In fact, it is possible to equivalently replace there the local convexity by the {\em global} one. Indeed, if there exists a locally convex function $\psi$ satisfying \eqref{varconvex}, we can define the function $\overline{\psi}:\R^n\to\overline{\R}$ by: 
		$$
		\overline{\psi}(x):=\begin{cases}
			\psi(x) &\text{if}\quad x\in U,\\
			\infty &\text{otherwise}.
		\end{cases}
		$$
		It is clear that $\overline{\psi}$ is an l.s.c.\ convex function on $\R^n$ with $\partial\overline{\psi}(x)=\partial\psi(x)$ for any $x\in U$, which implies that 
		$$
		(U_\epsilon\times V)\cap\gph\partial\varphi=(U\times V)\cap\gph\partial \overline{\psi}\;\text{and}\;\varphi(x)=\overline{\psi}(x)\; \text{at the common elements}\;(x,v). 
		$$
		Similarly, the local strong convexity of $\psi$ can be replaced by the global one for variational strong convexity. 
		
		{\bf(iii)} The {\em prox-regularity} of $\varphi$ at $\ox$ holds {\em automatically} if $\ph$ is {\em variationally convex} for this pair $(\ox,\ov)$. Indeed, we have by the convexity of $\psi$ in Definition~\ref{vr} that
		$$
		\psi(x)\ge\psi(u)+\langle v,x-u\rangle\quad\text{for all }\;x,u\in \R^n,\;v\in \partial\psi(u). 
		$$
		Combining this with \eqref{varconvex} and $\psi\le\ph$ on $U$ yields
		$$
		\varphi(x)\ge\varphi(u)+\langle v,x- u\rangle \quad\text{for all }\;x\in U,\;(u,v) \in\gph\partial\varphi\cap (U_\epsilon\times V),
		$$
		which tells us that $\varphi$ is prox-regular at $\ox$ for $\ov$. When $\varphi$ is subdifferentially continuous at $\ox$ for $\ov$, we can replace the set $U_\epsilon$ by the neighborhood $U$ in Definition~\ref{vr}.  
		
		{\bf(iv)} If $\varphi$ is {\em $\mathcal{C}^1$-smooth} around $\ox$, then the {\em variational} convexity (variational strong convexity) of $\varphi$ at $\ox$ for $\ov=\nabla\varphi(\ox)$ reduces to the {\em local} convexity (local strong convexity) around $\ox$. Indeed, the underlying condition \eqref{varconvex} reads for smooth functions $\ph$ as
		$$
		(U\times V)\cap\gph\nabla \varphi=(U\times V)\cap\gph\partial\psi\; \text{and}\; \varphi(x)=\psi(x)\;\text{at the common elements}\;(x,v).
		$$
		Due to the continuity of $\nabla\varphi$ around $\ox$, we can find a neighborhood $\tilde{U}\subset U$ of $\ox$ such that $\nabla\varphi(\tilde{U})\subset V$, which implies that $\varphi(x)=\psi(x)$ for any $x\in \tilde{U}$. Therefore, $\varphi$ is  convex (strongly convex) on $\tilde{U}$, which means the local convexity (local strong convex) of $\ph$ around $\ox$. This remark tells us that the novelty of variational convexity (variational strong convexity) vs.\ local convexity (local strong convexity) emerges only in {\em nonsmooth} functions.}
\end{Remark}

The following example presents a nonconvex variationally convex  function, which is called \textit{$\ell^0$ pseudo-norm}  and is commonly used in compressive sensing; see, e.g., \cite{cwb,dslmh}.

\begin{Example}[\bf variational convexity of $\ell^0$ pseudo-norm] \rm
	Let $\varphi:\R^n\to \R$ given by $\varphi(x):=\|x\|_0$, which is the $\ell^0$
	norm of $x$, counting the number of nonzero elements of $x$. In other words, we can represent $\varphi$ as 
	\begin{equation*}
		\varphi(x) = \sum_{i=1}^n I(x_i)\;\text{ for all }\; x=(x_1,\ldots,x_n)\in \R^n
	\end{equation*}
	by using the standard notation
	\begin{equation*}
		I(t):= \begin{cases}
			1&\text{if} \quad t \ne 0,\\
			0&\text{otherwise}.
		\end{cases}
	\end{equation*}
	It is easy to calculate the subdifferential of $\ph$ by
	$$
	\partial\varphi (x) =  \left\{v=(v_1,v_2,\ldots,v_n)\in\R^n\;\bigg|\;
	\begin{array}{@{}cc@{}}
		v_i= 0,\; \text{if } \; x_i\ne 0,\\
		v_i\in \R,\; \text{if } \; x_i=0
	\end{array}\right\}.
	$$
	Let us show that $\varphi$ is variationally convex at $\ox=0$ for any $\ov \in \partial\varphi(\ox)$. Indeed, take $\epsilon\in (0,1)$ and denote
	$$
	U:= \text{\rm int}\,\mathbb{B}_{\delta_1}(0)\times \ldots \times \text{\rm int}\,\mathbb{B}_{\delta_n}(0)\;\text{ and }\;V:= \text{\rm int}\,\mathbb{B}_\epsilon(\ov_1)\times\ldots \times  \text{\rm int}\,\mathbb{B}_\epsilon(\ov_n).
	$$
	where $\delta_i:=(1+|\bar{v}_i|)^{-1}$ for $i=1,\ldots,n$. Define further
	the l.s.c.\ convex function $\psi:\R^n\to\R$ by
	$$
	\psi(x):=  \|x\|_1 + \langle \ov, x\rangle\;\text{ for all }\;x \in \R^n. 
	$$
	It is clear that $\psi \leq \varphi$ on $U$ and that $U_\epsilon =\{0\}$.  
	Moreover, the simple computation tells us that
	$$
	\partial\psi (x) =  \left\{v=(v_1,\ldots,v_n)\in\R^n\;\bigg|\;
	\begin{array}{@{}cc@{}}
		v_i= \ov_i + \text{\rm sgn}(x_i),\; \text{if } \; x_i\ne 0,\\
		v_i\in \ov_i + [-1,1],\; \text{if } \; x_i=0
	\end{array}\right\},
	$$
	which implies therefore the equalities 
	$$
	(U_\epsilon\times V)\cap\gph\partial\varphi=(U\times V)\cap\gph\partial\psi = \{0\}\times V\;\text {and }\; \psi(0)=\varphi(0) 
	$$
	This justifies \eqref{varconvex} and thus shows that $\varphi$ is variationally convex at $\ox$ for $\ov$.  
\end{Example}

Next we recall some notions of monotonicity of set-valued mappings  and their local counterparts. {A multifunction $T:\R^n\rightrightarrows\R^n$ is} {\em monotone} relative to a subset $W\subset\R^n\times\R^n$ if
\begin{equation}\label{monotone}
	\langle v_1-v_2,u_1-u_2\rangle\ge 0\quad\text{for all }\;(u_1,v_1), (u_2,v_2) \in\gph T\cap W. 
\end{equation} 
Condition \eqref{monotone} is referred to as {\em global} monotonicity when $W=\R^n\times\R^n$. Similarly, {\em strong monotonicity} with modulus $\sigma>0$ corresponds to replacing $\langle v_1-v_2,u_1-u_2\rangle\ge 0$ in \eqref{monotone} by $\langle v_1-v_2,u_1-u_2 \rangle\ge\sigma\|u_1-u_2\|^2$.

Monotonicity can also be defined locally around a given point. More specifically, $T\colon\R^n\tto\R^n$ is {\em locally monotone} around $(\ox,\ov)\in\gph T$ if there exists a neighborhood $W\subset\R^n\times\R^n$ of $(\ox,\ov)$ such that $T$ is   monotone relative to this neighborhood.  For an l.s.c.\ function $\varphi:\R^n\to\overline{\R}$, the subgradient mapping $\partial\varphi:\R^n\rightrightarrows\R^n$ is said to be {\rm $\varphi$-locally monotone} around $(\ox,\ov)\in\gph\partial\varphi$ if there exists a convex neighborhood $U\times V$ of $(\ox,\ov)$ such that $\partial\varphi$ is monotone relative to the set 
\begin{equation}\label{W}
	W:=\big\{(x,y)\in U\times V\big|\;\varphi(x)<\varphi(\ox)+\epsilon\big\}
\end{equation} 
for some $\epsilon>0$. The same pattern defines the notion of {\em $\varphi$-local strong monotonicity}. 

It is clear that if $T:=\partial\varphi$ is locally monotone around $(\ox,\ov)\in \gph\partial\varphi$, then it is also $\varphi$-locally monotone around this point. The reverse implication is not correct. Indeed, consider the nonsmooth  l.s.c.\ function $\varphi:\R \to\R$ defined by $\varphi(x):=0$ if $x\le 0$ and $\varphi(x):=1$ if $x > 0$. It is easy to check that
$$
\widehat{\partial}\varphi(x)=\partial\varphi(x)= \begin{cases}
	\{0\} & \text{if} \quad x\ne 0,\\
	[0,\infty) & \text{otherwise}. 
\end{cases}
$$
We see that $\partial\varphi$ is not locally monotone around $(\ox,\ov):=(0,0)$, but it is $\varphi$-locally monotone around this point because $\partial\varphi$ is locally monotone relative to the set 
$$
W:=\big\{(x,y)\in\mathbb{B}_\epsilon(\ox)\times\mathbb{B}_\epsilon(\ov)\;\big|\; \varphi(x)<\varphi(\ox)+\epsilon\big\},
$$
where $\epsilon\in(0,1)$. When $\varphi$ is subdifferentially continuous at $\ox$ for $\ov$, there is no difference between the local monotonicity and $\varphi$-local monotonicity of $T:=\partial\varphi$ around $(\ox,\ov)$.

\medskip 
Next we consider a continuous  
(i.e., certainly being subdifferentially continuous)  and variationally convex function on $\R$, which is known as the \textit{log-sum penalty function}. This function is commonly used to bridge the gap between the $\ell_0$ and $\ell_1$ norms in compressive sensing \cite{cwb} and as a nonconvex
surrogate function of the matrix rank function in the low-rank regularization \cite{dslmh}.

\begin{Example}[\bf log-sum penalty function]\rm Let  $\varphi:\R^n\to \R$ given by 
	$$
	\varphi(x)=  \sum_{i=1}^n  \log (1+|x_i|)\;\text {for all }\; x=(x_1,\ldots,x_n) \in \R^n.
	$$ 
	We clearly have the subdifferential representation 
	$$
	\partial\varphi (x) =\big\{v=(v_1,v_2,\ldots,v_n)\in\R^n\;\big|\; v_i \in G(x_i), \; i=1,\ldots, n\big\}, 
	$$
	where the multifunction $G: \R \rightrightarrows \R$ is defined by
	$$
	G(t):= \begin{cases}
		\displaystyle \frac{1}{t-1} & \text{if }\quad t< 0,\\
		[-1,1]&\text{if }\quad t =0,\\
		\displaystyle \frac{1}{t+1} &\text{if }\quad t>0.
	\end{cases}
	$$
	To show that $\varphi$ is variationally convex at $\ox:=0$ for any $\ov \in \text{\rm int}\,\partial\varphi(\ox)=(-1,1)^n$, observe that $\partial \varphi$ is locally monotone around $(\ox,\ov)$. Then the claimed variational convexity of $\ph$ follows from \cite[Theorem~1]{r19}.
\end{Example}

\color{black}
Finally in this section, we recall the fundamental notion of {\em tilt stability} of 
local minimizers introduced by Poliquin and Rockafellar \cite{Poli} and then comprehensively investigated in many publications mentioned above (see also the references therein) with numerous applications in variational analysis and optimization.

\begin{Definition}[\bf tilt-stable local minimizers]\label{def:tilt} \rm Given $\varphi\colon\R^n\to\oR$, a point $\ox\in\dom\varphi$ is a {\it tilt-stable local minimizer} of $\varphi$ if there exists a number $\gamma>0$ such that the mapping
	\begin{equation}\label{tilt} 
		M_\gamma\colon v\mapsto{\rm
			argmin}\big\{\varphi(x)-\langle v,x\rangle\;\big|\;x
		\in\B_\gamma(\ox)\big\} 
	\end{equation}
	is single-valued and Lipschitz continuous on some neighborhood of $0\in\R^n$ with
	$M_\gamma(0)=\{\ox\}$. 
\end{Definition}\vspace*{-0.03in} 
We also consider in what follows a {\em quantitative} version of this notion that specifies a modulus of tilt stability. Namely, $\ox$ is a tilt-stable minimizer of $\varphi$ with {\em modulus} $\kappa>0$ if the mapping $M_\gamma$ is Lipschitz continuous with constant $\kappa$ in the framework of Definition~\ref{def:tilt}.\vspace*{0.05in}

The remark below discusses relationships between variational convexity and local minimizers as well as between variational strong convexity and tilt-stable local minimizers.

\begin{Remark}[\bf variational convexity and local minimizers]\label{sttilt} \rm Considering an extended-real-valued l.s.c.\ function $\varphi:\R^n\to\overline{\R}$, observe the following:
	
	{\bf(i)} If $\varphi:\R^n\to\overline{\R}$ is variationally convex at $\ox$ for $0\in\partial\varphi(\ox)$, then $\ox$ is a local minimizer of $\varphi$. The {\em reverse} implication {\em fails} in general. Indeed, consider the continuous function 
	\begin{equation}\label{exa1}
		\varphi(x):=\begin{cases}
			\text{\rm min}\left\{\disp\left(1+\frac{1}{n}\right)|x|-\frac{1}{n(n+1)},\frac{1}{n}\right\} &\text{if}\quad\disp\frac{1}{n+1}\leq |x|\leq \frac{1}{n},\\
			0&\text{if}\quad x=0
		\end{cases}
	\end{equation}
	taken from \cite[Example~3.4]{dl}. It is easy to check that this function is not prox-regular at its local minimizer $\ox=0$ for $\ov=0$, and hence it is not variationally convex for these points by Remark~\ref{disvar}(iii).
	
	{\bf(ii)} It follows from \cite[Theorems~2.3]{r19} that if $\varphi:\R^n\to\overline{\R}$ is variationally strongly convex at $\ox$ for $0\in\partial\varphi(\ox)$ with modulus $\sigma>0$, then $\ox$ is a tilt-stable local minimizer of $\ph$ with modulus $\sigma^{-1}$. However, the {\em reverse} implication may {\em fail} in simple situations. As shown in \cite[Example~3.4]{dl}, the point $\ox=0$ is a tilt-stable local minimizer of $\ph$ from \eqref{exa1}. However, by (i) this function is not even variationally convex at $\ox$ for $\ov=0$.
\end{Remark}

Note that the failure of the reverse implication in Remark~\ref{sttilt}{(ii) is {\em impossible} if a function $\ph\colon\R^n\to\oR$  is {\em continuously prox-regular} at $\ox$ for $0$. This is proved in the next proposition.
	\begin{Proposition}[\bf equivalent descriptions of variational strong convexity]\label{equitiltstr} Let $\varphi:\R^n\to\overline{\R}$ be continuously prox-regular at $\bar{x}\in\dom\varphi$ for $0\in\partial\varphi(\bar{x})$. Then the following assertions are equivalent:
		
		{\bf(i)}  $\varphi$ is variationally strongly convex at $\ox$ for $\ov=0$ with modulus $\sigma>0$. 
		
		{\bf (ii)} $\ox$ is a tilt-stable local minimizer of $\varphi$ with modulus $\sigma^{-1}$. 
		
		{\bf (iii)} $\partial \varphi$ is locally strongly monotone around $(\ox,0)$ with modulus $\sigma >0$. 
	\end{Proposition}
	\begin{proof} Due to the continuous prox-regularity of $\varphi$ at $\ox$ for $0$ and \cite[Lemma~\ref{strongconvex}]{r19}, the variational strong convexity of $\varphi$ at $\ox$ for $0$ with modulus $\sigma >0$ is equivalent to the existence of neighborhoods $U$ of $\ox$ and $V$ of $0$ such that 
		$$
		\varphi(x)\ge\varphi(u)+\langle v,x-u\rangle+\frac{\sigma}{2} \|x-u\|^2\;\mbox{ for all }\;x\in U,\;(u,v)\in\gph\partial\varphi\cap(U\times V).
		$$
		The latter condition is equivalent to $\ox$ being a tilt-stable local minimizer of $\varphi$ with modulus $\sigma^{-1}$ by \cite[Theorem~3.2]{MorduNghia}, which justifies therefore the equivalence between (i) and (ii). The equivalence between the assertions (i) and (iii) follows from \cite[Theorem 2]{r19}.  
	\end{proof}
	
	\begin{Remark}[\bf on second-order characterizations of variational strong convexity] \rm It is known that the tilt stability of a continuously prox-regular function $\varphi$ at $\ox$ for $0$ is equivalent to the positive-definiteness of the second-order subdifferential $\partial^2\varphi(\ox,0)$; see \cite{Poli}. Therefore, by the equivalent descriptions of the variational strong convexity in Proposition \ref{equitiltstr}, we can immediately obtain the second-order characterizations for variational strong convexity at $\ox$ for $\ov=0$ in the case where $\varphi$ is continuously prox-regular. It is shown in Section~\ref{sec:codcharacterizationconvex} that we can obtain the coderivative-based characterizations of both variational convexity and variational strong convexity {\em without} using the characterizations of tilt stability. This new approach is mainly based on the usage of the ``hidden convexity" of Moreau envelopes established in the next section.
	\end{Remark}
\vspace*{-0.15in}

\color{black}
\section{Variational Convexity via Moreau Envelopes}\label{sec:Moreau}\vspace*{-0.05in}
\setcounter{equation}{0}

As mentioned in Section~\ref{sec:prelim}, a variationally convex function around a given point is not necessarily convex around this point. In this section, we  show that the variational convexity of an l.s.c.\ function is {\em equivalent} to  the local convexity of its {\em Moreau envelope}.  This kind of ``hidden convexity" is not only significant for its own sake, but also plays a crucial role in establishing the new characterizations of variational convexity via second-order subdifferentials \eqref{limitsec} and \eqref{seccombine}, which are provided in Section~\ref{sec:codcharacterizationconvex}.\vspace*{0.03in}

To proceed, we begin with the following lemma taken from \cite[Theorem 1]{r19}, which lists the characterizations of variationally convex functions via some properties of  {their limiting subdifferentials \eqref{MordukhovichSubdifferential}.}

\begin{Lemma}[\bf subgradient characterizations of variational convexity]\label{monoviaMoreau}	Let $\varphi:\R^n\to\overline{\R}$  be an l.s.c.  function with $\bar{x}\in\dom\varphi$ and $\bar{v}\in\widehat{\partial}\varphi(\bar{x})$. The  following assertions are equivalent:

{\bf(i)} $\varphi$ is variationally convex at $\ox$ for $\ov$.

{\bf(ii)} $\partial\varphi$ is $\varphi$-locally monotone around $(\bar{x},\bar{v})$.

{\bf(iii)} There are neighborhoods $U$ of $\bar{x}$, $V$ of $\bar{v}$, and a number $\epsilon>0$ such that 
\begin{equation}\label{convexinequa}
\varphi(x)\ge\varphi(u)+\langle v, x-u\rangle\quad\text{for all }\;x\in U,\;(u,v)\in\gph\partial\varphi\cap (U_\epsilon\times V),
\end{equation}
where $U_\epsilon:=\{u\in U\;|\;\varphi(x)<\varphi(\ox)+\epsilon\}$. 
\end{Lemma}

Here is the aforementioned interconnection of variational convexity of extended-real-valued functions and local convexity of their Moreau envelopes. 

\begin{Theorem}[\bf characterization of variational convexity via Moreau envelopes] \label{convexviaMo} Let $\varphi:\R^n\to\overline{\R}$  be an l.s.c.\ and prox-bounded function with $\bar{x}\in\dom\varphi$ and $\bar{v}\in\partial\varphi(\bar{x})$. The following assertions are equivalent: 

{\bf(i)} $\varphi$ is variationally convex at $\ox$ for $\ov$.

{\bf(ii)} $\varphi$ is prox-regular at $\ox$ for $\ov$, and the Moreau envelope $e_\lambda \varphi$ is locally convex around $\bar{x}+\lambda\bar{v}$ for small $\lambda>0$. 
\end{Theorem}
\begin{proof} By Remark~\ref{disvar}(iii), the variational convexity of a function yields its prox-regularity. Therefore, we only need to verify that assertion  (i) is equivalent to the fact that the Moreau envelope $e_\lambda\varphi$ is locally convex around $\ox+\lambda\ov$ for all $\lm>0$ sufficiently small, provided that $\varphi$ is prox-regular at $\ox$ for $\ov$. Observe first that it follows directly from definitions \eqref{prox} of prox-regularity and \eqref{FrechetSubdifferential} of regular subgradients that $\ov\in\widehat{\partial}\varphi(\ox)$. Proposition~\ref{C11} tells us that the assumed lower semicontinuity, prox-boundedness, and prox-regularity properties of $\ph$ guarantee the existence of $\lambda_0>0$ and a $\varphi$-attentive $\gamma$-localization $T:\R^n\rightrightarrows\R^n$ given by
\begin{equation}\label{att} 
T(x):=\begin{cases}\big\{v\in\partial\varphi(x)\big|\;\|v-\ov\|<\gamma\big\} & \text{if}\quad\|x-\ox\|<\gamma\;\text{ and }\;\varphi(x)<\varphi(\ox)+\gamma,\\
\emp &\text{otherwise}
\end{cases}
\end{equation} 
such that for any $\lambda\in(0,\lambda_0)$ there is a neighborhood $U_\lambda$ of  $\bar{x}+\lambda\bar{v}$ on which $e_\lambda\varphi$ is of class $\mathcal{C}^{1,1}$, that the proximal mapping $\text{\rm Prox}_{\lambda\varphi}$ is single-valued and Lipschitz continuous on $U_\lambda$ with ${\rm Prox}_{\lm\ph}(\ox+\lm\ov)=\ox$, and that the gradient expression for the Moreau envelope \eqref{GradEnvelope} holds with $T$ taken from \eqref{att}.

To verify (i)$\Longrightarrow$(ii), suppose that $\varphi$ is variationally convex at $\ox$ for $\ov$ and deduce from Lemma~\ref{monoviaMoreau} that $\partial\varphi$ is  $\varphi$-locally monotone around $(\bar{x},\bar{v})$. Then there exists $\epsilon>0$  such that the subgradient mapping $\partial\varphi$ is monotone relative to the set $W$ from \eqref{W}. Fixing any $\lambda\in(0,\lambda_0)$, we get from  \eqref{GradEnvelope} that
\begin{equation}\label{I-e}
(I-\lambda\nabla e_\lambda\varphi)(\ox+\lambda \ov)=\ox\quad\text{and }\;\nabla e_\lambda\varphi(\ox+\lambda\ov)=\ov.
\end{equation} 
It follows from \eqref{I-e}  and from the continuity of the mappings $\psi_\lambda:=e_\lambda\varphi -\frac{1}{2\lambda}\|\text{\rm Prox}_{\lambda\varphi}(\cdot)-\cdot \|^2$ and $\nabla  e_\lambda\varphi$, $I-\lambda\nabla e_\lambda\varphi$ around $\bar{x}+\lambda \bar{v}$ that there exists a neighborhood $U\subset U_\lambda$ of $\ox+\lambda \ov$ such that
\begin{equation}\label{Tcont}
\big((I-\lambda\nabla e_\lambda\varphi)(x)\times\nabla e_\lambda\varphi(x)\big)\in \mathbb{B}_\epsilon(\ox)\times \mathbb{B}_\epsilon(\ov)\quad\text{and }\; \psi_\lambda(x)-\psi_\lambda(\ox+\lambda\ov)<\epsilon\quad \text{for all }\;x\in U. 
\end{equation}
Definitions \eqref{Moreau}, \eqref{ProxMapping} and the above construction of $\psi_\lm$ ensure that $\psi_\lambda(x)=\varphi(\text{\rm Prox}_{\lambda\varphi}(x) )$ for all $x\in U$. Fixing now any $x_1, x_2 \in U$ and using \eqref{GradEnvelope} and \eqref{Tcont}, we get for $i=1,2$ the relationships
\begin{equation*}
\nabla e_\lambda\varphi(x_i)\in T\big(x_i-\lambda\nabla e_\lambda \varphi(x_i)\big),
\end{equation*}
\begin{equation*}
\begin{array}{ll}
\varphi\big(x_i-\lambda\nabla e_\lambda\varphi(x_i)\big)=\varphi\big((\text{\rm Prox}_{\lambda\varphi}(x_i)\big)=\psi_\lambda(x_i)\\
<\psi_\lambda(\ox+\lambda\ov)+\epsilon=\varphi\big(\text{\rm Prox}_{\lambda \varphi}(\ox+\lambda\ov)\big)+\epsilon=\varphi(\ox)+\epsilon.
\end{array}
\end{equation*}
By \eqref{W}, this justifies the inclusions $(x_i-\lambda\nabla e_\lambda\varphi(x_i), \nabla e_\lambda\varphi(x_i))\in W\cap\gph T$ for $i=1,2$. Combining the latter with the monotonicity of $\partial\varphi$ relative to $W$, we arrive at 
$$
\big\langle\nabla e_\lambda\varphi(x_1)-\nabla e_\lambda\varphi(x_2),x_1- x_2\big\rangle\ge\lambda\|\nabla e_\lambda\varphi(x_1)-\nabla e_\lambda\varphi(x_2)\|^2\ge 0, 
$$
which verifies the convexity of $e_\lambda\varphi$ on $U$ due to \cite[Theorem~4.1.4]{hl04}, and therefore (ii) holds.\vspace*{0.03in} 

To prove next the reverse implication (ii)$\Longrightarrow$(i), assume that the Moreau envelope $e_\lambda \varphi$ is locally convex around $\ox+\lambda\ov$ for $\lm>0$ sufficiently small. Fixing $\lambda\in(0,\lambda_0)$, suppose without loss of generality that $e_\lambda\varphi$ is convex on $U_\lambda$. Utilizing the first-order characterization of ${\cal C}^1$-smooth convex functions in \cite[Theorem~4.1.1]{hl04} gives us
\begin{equation}\label{convexMoreau}
e_\lambda\varphi(x)\geq e_\lambda\varphi(u)+\langle\nabla e_\lambda\varphi(u),x-u\rangle\;\mbox{ for all }\;x,u\in U_\lambda. 
\end{equation} 
Select neighborhoods $U$  of $\ox$ and $V$ of $\bar{v}$ such that $U\subset{\rm int}\,\mathbb{B}_\gamma(\ox)$, $V\subset{\rm int}\,\mathbb{B}_\gamma (\ov)$, and 
\begin{equation}\label{neighbor1}
x+\lambda v\in U_\lambda\;\mbox{ whenever }\;x \in  U\;\mbox{ and }\;v\in V. 
\end{equation}
By Lemma~\ref{monoviaMoreau}, it suffices to verify that 
\begin{equation}\label{convexineq}
\varphi(x)\geq \varphi(u)+\langle v,x-u \rangle\;\mbox{ for all }\;x \in U,\;(u,v) \in\gph\partial\varphi\cap(U_\gamma\times V),
\end{equation}
where $U_\gamma$ as defined as in Lemma~\ref{monoviaMoreau}. Pick $x\in U$, $(u,v) \in\gph\partial\varphi\cap(U_\gamma\times V)$ and deduce from \eqref{att} that $(u,v) \in\gph T$, which yields $u+\lambda v\in(\lambda I+T^{-1})(v)$. It follows from \eqref{GradEnvelope} and \eqref{neighbor1} that  
$$
\nabla e_\lambda\varphi(u+\lambda v)=(\lambda I+T^{-1})^{-1}(u+\lambda v)= v,
$$
$$
\text{\rm Prox}_{\lambda\varphi}(u+\lambda v)=u+\lambda v-\lambda\nabla e_\lambda\varphi(u+\lambda v)=u,
$$
which implies in turn the equalities 
\begin{equation*}
\begin{array}{ll}
e_\lambda\varphi(u+\lambda v)=\varphi\big(\text{\rm Prox}_{\lambda\varphi}(u+\lambda v)\big)+\disp\frac{1}{2\lambda}\|u+\lambda v-\text{\rm Prox}_{\lambda\varphi}(u+\lambda v)\|^2\\
=\varphi(u)+\disp\frac{1}{2\lambda}\|u+\lambda v-u\|^2
=\varphi(u)+\frac{\lambda}{2}\|v\|^2.
\end{array}
\end{equation*}
By \eqref{convexMoreau}, we get the relationships
\begin{eqnarray*}
\begin{array}{ll}
\varphi(x)+\disp\frac{\lambda}{2}\|v\|^2=\varphi(x)+\frac{1}{2\lambda}\|x+\lambda v -x\|^2\ge e_\lambda\varphi(x+\lambda v)\\
\ge e_\lambda\varphi(u+\lambda v)+\langle\nabla e_\lambda\varphi(u+\lambda v), x+\lambda v-(u+\lambda v)\rangle=\varphi(u)+\disp\frac{\lambda}{2}\|v\|^2+\langle v, x-u\rangle.
\end{array}
\end{eqnarray*}
Subtracting the term $(\lm/2)\|v\|^2$ from both sides of the above inequalities, we arrive at the desired condition \eqref{convexineq} and thus completes the proof of the theorem. 
\end{proof}

Finally in this section, we discuss some applications of Theorem~\ref{convexviaMo} to {\em numerical optimization}.

\begin{Remark}[\bf applications of variational convexity to generalized Newton methods]\label{vc-newton}
{\rm In the recent papers \cite{BorisKhanhPhat,kmptjogo,kmptmp} and \cite{BorisEbrahim}, several {\em generalized Newton methods} of nonsmooth optimization have been designed and justified in the following pattern: dealing first with problems of {\em ${\cal C}^{1,1}$ optimization} and then propagating the algorithms and the imposed assumptions to {\em extended-real-valued prox-regular functions} by using their Moreau envelopes and proximal mappings. In this way, well-posedness and superlinear local and global convergence of the proposed generalized Newton algorithms have been established for problems of minimizing prox-regular functions, and hence for problems of constrained optimization, in terms of their given data as well as of solutions to subproblems in \eqref{Moreau} and/or \eqref{ProxMapping}. Although constructive applications of this procedure to solving some practical models, which appear in machine learning, statistics, etc., are developed in the aforementioned papers, the implementation of the designed algorithms for general classes of prox-regular functions is a challenging issue. However, the {\em variational convexity} of the original cost function allows us to reduce, by Theorem~\ref{convexviaMo}, the minimization subproblems in \eqref{Moreau} and \eqref{ProxMapping} to problems of {\em convex optimization}, which admit the much more elaborated machinery to be constructively resolved. Detailed investigations in this direction will be conducted in our future research.}
\end{Remark}
\vspace*{-0.22in}  

\section{Variational Strong Convexity via Moreau Envelopes}\label{sec:Moreau1}\vspace*{-0.05in}\setcounter{equation}{0}

The main goal of this section is to establish characterizations of {\em variational strong convexity} of extended-real-valued functions via {\em strong convexity} of Moreau envelopes with precise relationships between the corresponding {\em moduli}.\vspace*{0.05in}

We begin with recalling subgradient characterizations of variational strong convexity taken from \cite[Theorem~2]{r19}.\vspace*{-0.03in}

\begin{Lemma}[\bf subgradient characterizations of variational strong convexity]\label{strongconvex} Let $\varphi:\R^n\to\overline{\R}$ be an l.s.c.\  function with $\bar{x}\in\dom\varphi$ and $\bar{v}\in\widehat{\partial}\varphi(\bar{x})$. The following assertions are equivalent:

{\bf(i)} $\varphi$ is  variationally strongly convex at $\ox$ for $\ov$ with modulus $\sigma>0$. 

{\bf(ii)} $\partial\varphi$ is $\varphi$-locally strongly monotone around $(\bar{x},\bar{v})$ with modulus $\sigma>0$.

{\bf(iii)} There exist neighborhoods $U$ of $\bar{x}$, $V$ of $\bar{v}$, and a number $\epsilon>0$ such that 
\begin{equation}\label{convexinequa2}
\varphi(x)\ge\varphi(u)+\langle v,x-u\rangle +\frac{\sigma}{2}\|x-u\|^2\quad \text{for all }\;x\in U,\;(u,v)\in\gph\partial\varphi\cap(U_\epsilon\times V),
\end{equation}
where $U_\epsilon$ is defined in Lemma~{\rm\ref{monoviaMoreau}}.
\end{Lemma}
		
Given an l.s.c.\ function $\varphi:\R^n\to \overline{\R}$ with $\bar{x}\in\dom\varphi$ and $\sigma\ne 0$, consider its {\em $\sigma$-quadratic shift} 
\begin{equation}\label{shifted}
\vt(x):=\varphi(x)-\frac{\sigma}{2}\|x-\ox\|^2\quad\text{for all }\;x\in\R^n.
\end{equation} 

The next lemma presents relationships between variational strong convexity of a function and variational convexity of its quadratic shift \eqref{shifted}. 

\begin{Lemma}[\bf variational strong convexity via quadratic shifts]\label{cvandstrongcv} Let $\varphi:\R^n\to\overline{\R}$ be an l.s.c. function with $\bar{x}\in\dom\varphi$ and $\bar{v}\in\partial\varphi(\bar{x})$, and let $\vt$ be defined in \eqref{shifted} with some $\sigma>0$. Then the following hold:

{\bf(i)} $\partial\varphi$ is $\varphi$-locally strongly monotone around $(\bar{x},\bar{v})$ with modulus $\sigma$ if and only if $\partial\vt$ is $\vt$-locally monotone around the pair $(\bar{x},\bar{v})$. 

{\bf(ii)} $\varphi$ is variationally strongly convex at $\ox$ for $\ov$ with modulus $\sigma$ if and only if the quadratically shifted function $\vt$ is variationally convex at $\ox$ for $\ov$. 
\end{Lemma} 
\begin{proof} The elementary {sum rule for limiting subdifferentials} from  \cite[Proposition~1.30]{Mor18} gives us 
\begin{equation*}
\partial\vt(x)=\partial\varphi(x)-\sigma(x-\bar{x}) \quad\text{\rm whenever }\;x\in\dom\varphi.
\end{equation*}	
Therefore, we have $\bar{v}\in\partial\vt(\bar{x})$ if and only if $\ov\in\partial\ph(\ox)$. Since (ii) follows directly from (i) by using the equivalences in Lemmas~\ref{monoviaMoreau} and \ref{strongconvex}, it suffices to verify (i). To furnish this, fix $\varepsilon>0$ and denote $U:={\rm int}\,\mathbb{B}_{\sqrt{\epsilon/\sigma}}(\ox)$. Define
$$
U_\varepsilon^\varphi:=\big\{x\in U \;\big|\;\varphi(x)<\varphi(\bar{x})+\varepsilon\big\},
$$
$$
U_{\varepsilon/2}^\vt:=\big\{x\in U\;\big|\; \vt(x)<\vt(\bar{x})+\varepsilon/2\big\},\;\text {and }\;U_{\varepsilon}^\vt:=\big\{x\in U \;\big|\;\vt(x)<\vt(\bar{x})+\varepsilon\big\}.
$$
Given a convex neighborhood $V$ of $\ov$, both sets $V+\sigma(U-\bar{x})$ and $V-\sigma(U-\bar{x})$ are also convex neighborhoods of $\bar{v}$ with the fulfillment of implications
\begin{equation}\label{FirstImplication}
(x,v)\in\gph\partial\vt\cap(U_{\epsilon/2}^\vt\times V)\Longrightarrow\big(x,v+\sigma (x-\ox)\big)\in\gph\partial\varphi\cap\big(U_\varepsilon^\varphi\times[V
+\sigma(U-\bar{x})]\big),
\end{equation}
\begin{equation}\label{SecondImplication}
(x,v)\in\gph\partial\varphi\cap(U_\varepsilon^\varphi\times V)\Longrightarrow \big(x,v-\sigma(x-\ox)\big)\in\gph\partial\vt\cap\big(U_\varepsilon^\vt\times [V-\sigma(U-\bar{x})]\big).
\end{equation}
Furthermore, we have the following equivalences:
\begin{equation}\label{FirstEquivalence}
\langle v_1-v_2,x_1-x_2\rangle\ge 0\Longleftrightarrow 
\langle [v_1+\sigma(x_1-\bar{x})]-[v_2+\sigma(x_2-\bar{x})],x_1-x_2\rangle\ge \sigma\|x_1-x_2\|^2,
\end{equation}
\begin{equation}\label{SecondEquivalence}
\langle v_1-v_2,x_1-x_2\rangle\geq \sigma\|x_1-x_2\|^2\Longleftrightarrow 
\langle[v_1-\sigma(x_1-\bar{x})]-[v_2-\sigma(x_2-\bar{x})],x_1-x_2\rangle\ge 0.
\end{equation}
Combining \eqref{FirstImplication} and \eqref{FirstEquivalence} gives us the implication
\begin{center}
$\varphi$-locally strong monotonicity of $\partial\varphi$ around $(\bar{x},\bar{v})$ $\Longrightarrow$ $\vt$-local monotonicity of $\partial\vt$ around $(\bar{x},\bar{v})$.
\end{center}
Finally, combining \eqref{SecondImplication} and \eqref{SecondEquivalence} tells us that
\begin{center}
$\vt$-local monotonicity of $\partial\vt$ around $(\bar{x},\bar{v})$ $\Longrightarrow$ $\varphi$-locally strong monotonicity of $\partial\varphi$ around $(\bar{x},\bar{v})$,
\end{center}
which therefore completes the proof of the lemma.
\end{proof}
		
The next lemma calculated the Moreau envelope of the quadratically shifted function \eqref{shifted}.
		
\begin{Lemma}[\bf Moreau envelope of quadratic shifts]  Let $\varphi:\R^n\to\overline{\R}$ be an l.s.c. and prox-bounded function, and let $\ox\in\dom\varphi$. Given $\sigma\ne 0$, consider the quadratic shift \eqref{shifted}.
Then for each $\gamma\in (0,|\sigma|^{-1})$ we have 
\begin{equation}\label{envelopepsi}
e_\gamma\vt(x)=e_{{\gamma}/(1-\sigma\gamma)}\varphi\left(\frac{x-\sigma\gamma \ox}{1-\sigma\gamma}\right)-\frac{\sigma}{2(1-\sigma\gamma)}\|x-\ox\|^2,\quad x\in\R^n.
\end{equation} 
Consequently, the following holds for any $\lambda\in(0,|\sigma|^{-1})$:
\begin{equation}\label{envelopevarphi}
e_\lambda\varphi(x)= e_{\lambda/(1+\sigma\lambda)}\vt\left(\frac{x+\sigma\lambda\ox}{1+\sigma\lambda}\right)+\frac{\sigma}{2(1+\sigma\lambda)}\|x-\ox\|^2,\quad x\in\R^n. 
\end{equation}
\end{Lemma}
\begin{proof} Pick any $\gamma\in(0,|\sigma|^{-1})$ and get by definition \eqref{Moreau} of the Moreau envelope that
\begin{align*}
e_\gamma\vt(x)&=\inf_{y\in\R^n}\left\{\vt(y)+\frac{1}{2\gamma}\|y-x\|^2 \right\}=\inf_{y\in\R^n}\left\{\varphi(y)-\frac{\sigma}{2}\|y-\ox\|^2+\frac{1}{2\gamma}\|y-x\|^2\right\}\\
&= \inf_{y\in\R^n}\left\{\varphi(y)-\frac{\sigma}{2}\|y-\ox\|^2+\frac{1}{2\gamma}\|y-\ox\|^2 +\frac{1}{\gamma}\langle y-\ox,\ox -x \rangle  +\frac{1}{2\gamma}\|x-\ox\|^2 \right\}\\
&=\inf_{y\in\R^n}\left\{\varphi(y)+\frac{1-\sigma\gamma}{2\gamma}\|y-\ox\|^2 +\frac{1}{\gamma}\langle y-\ox,\ox -x \rangle  +\frac{1}{2\gamma(1-\sigma\gamma)}\|x-\ox\|^2 + \left(\frac{1}{2\gamma}-\frac{1}{2\gamma(1-\sigma\gamma)}  \right)\|\ox-x\|^2 \right\}\\
&= \inf_{y\in\R^n}\Bigg\{\varphi(y)+\frac{1}{2\left(\frac{\gamma}{1-\sigma\gamma} \right)}\Big\|y-\frac{x-\sigma\gamma\ox}{1-\sigma\gamma} \Big\|^2\Bigg\} -\frac{\sigma}{2(1-\sigma\gamma)}\|x-\ox\|^2\\
&=e_{\gamma/(1-\sigma\gamma)}\varphi\left(\frac{x-\sigma\gamma\ox}{1-\sigma\gamma} \right)-\frac{\sigma}{2(1-\sigma\gamma)}\|x-\ox\|^2,\quad x\in\R^n,
\end{align*}
which verifies \eqref{envelopepsi}. The second representation \eqref{envelopevarphi} follows from \eqref{envelopepsi} since $\varphi$ can be seen as the shifted function of $\vt$ with modulus $-\sigma$, i.e., $\varphi(x)=\vt(x)-(-\sigma/2)\|x-\ox\|^2$.  
\end{proof}
		
Now we are ready to derive the main result of this section that establishes the equivalence between the variational strong convexity of an extended-real-valued function and the local strong convexity of its Moreau envelope with a precise relationship between the corresponding moduli.

\begin{Theorem}[\bf quantitative characterization of variational strong convexity via Moreau envelopes]\label{stronMore}  Let $\varphi:\R^n\to\overline{\R}$ be l.s.c.\ and prox-bounded with $\bar{x}\in\dom\varphi$ and $\bar{v}\in\partial\varphi(\bar{x})$. Then the following are equivalent:

{\bf(i)} $\varphi$ is variationally strongly convex at $\ox$ for $\ov$ with modulus $\sigma>0$.

{\bf(ii)} $\varphi$ is prox-regular at $\ox$ for $\ov$ and $e_\lambda \varphi$ is locally strongly convex around $\ox+\lambda \ov$ with modulus $\frac{\sigma}{1+\sigma\lambda}$ for all numbers $\lm>0$ sufficiently small. 
\end{Theorem}
\begin{proof}  
We start with verifying implication (i)$\Longrightarrow$(ii). Note that the claimed prox-regularity of $\ph$ under (i) is checked in Remark~\ref{disvar}(iii). Considering the shifted function $\vt$ from\eqref{shifted}, we get by Lemma~\ref{cvandstrongcv} that $\vt$ is variationally convex at $\ox$ for $\ov$.
Utilizing Theorem~\ref{convexviaMo} tells us that $e_\lambda\vt$ is locally convex around $\ox+\lambda\ov$ for all small $\lambda>0$. Fix such a number $\lambda$ and define $\gamma:=\lambda/(1+\sigma\lambda)$. Since $\gamma<\lambda$, we find $\eta>0$ ensuring that $e_\gamma\vt$ is convex on $\mathbb{B}_\eta(\ox+\gamma\ov)$. Letting $\epsilon:=(1+\sigma\lambda)\eta>0$ gives us the implication
$$
x\in\mathbb{B}_{\epsilon}(\ox+\lambda \ov)\Longrightarrow\frac{x+\lambda\sigma \ox}{1+\sigma\lambda}\in\mathbb{B}_\eta(\ox+\gamma\ov).
$$
Define further the function $\th:\mathbb{B}_{\epsilon}(\ox+\lambda\ov)\to\mathbb{B}_\eta(\ox+\gamma\ov)$ by
\begin{equation}\label{th}
\th(x):=\frac{x+\lambda\sigma\ox}{1+\sigma\lambda}
\end{equation}
and deduce from \eqref{envelopevarphi} that the following representation of the Moreau envelope:
$$
e_\lambda\varphi=(e_\gamma\vt)\circ\th+\frac{\sigma}{2(1+\sigma\lambda)}\|\cdot -\ox\|^2.
$$ 
Since $(e_\gamma\vt)\circ\th$ is a composition of a convex function and an affine one, it is convex on $\mathbb{B}_{\epsilon}(\ox+\lambda\ov)$. This tells us that $e_\lambda\varphi$ is strongly convex on $\mathbb{B}_{\epsilon}(\ox+\lambda \ov)$ with modulus $\sigma/(1+\sigma\lambda)$, which justifies the claimed implication.\vspace*{0.03in} 

Next we verify implication (ii)$\Longrightarrow$(i). Suppose that the Moreau envelope $e_\lambda \varphi$ is locally strongly convex around $\ox+\lambda\ov$ with modulus $\mu(\lambda):=\sigma/(1+\sigma\lambda)$ for all $\lambda\in(0,\sigma^{-1})$ sufficiently small. Fixing such a number $\lambda$, for each $\gamma\in(0,\lambda/(1+\sigma\lambda))$ we have $0<\gamma/(1-\sigma\gamma)<\lambda$. This allows us to choose $\eta>0$ so that $e_{\gamma/(1-\sigma\gamma)}\varphi$ is strongly convex on $\mathbb{B}_\eta \left(\ox+\frac{\gamma}{1-\sigma\gamma}\ov\right)$ with modulus $\mu(\gamma/(1-\sigma\gamma))=\sigma (1-\sigma\gamma)$. Letting $\epsilon:=(1-\sigma\gamma)\eta>0$ yields the implication
$$
x\in\mathbb{B}_{\epsilon}(\ox+\gamma \ov)\Longrightarrow\frac{x-\sigma\gamma\ox}{1-\sigma\gamma}\in\mathbb{B}_\eta \left(\ox+\frac{\gamma}{1-\sigma\gamma}\ov\right).
$$
Symmetrically to \eqref{th}, define the function $\tilde\th:\mathbb{B}_{\epsilon}(\ox+\gamma\ov)\to\mathbb{B}_\eta\left(\ox+\frac{\gamma}{1-\sigma\gamma}\ov\right)$ by
$$
\tilde\th(x):=\frac{x-\gg\sigma\ox}{1-\sigma\gamma}
$$
and then deduce from \eqref{envelopepsi} the representation
$$
e_\gamma\vt=(e_{\gamma/(1-\sigma\gamma)}\varphi)\circ\tilde\th-\frac{\sigma}{2(1-\sigma\gamma)}\|\cdot-\ox\|^2,
$$
where $\vartheta$ is taken from\eqref{shifted}. As follows from the discussion after the definition of strong convexity in \eqref{ineqstrongfunction}, the composite function $(e_{\gamma/(1-\sigma\gamma)}\varphi)\circ\tilde\th$ is strongly convex on $\mathbb{B}_{\epsilon}(\ox+\gamma\ov)$ with modulus 
$$
\sigma(1-\sigma\gamma) \left(\frac{1}{1-\sigma\gamma}\right)^2=  \frac{\sigma}{1-\sigma\gamma}. 
$$
This tells us that $e_\gamma\vt$ is convex on $\mathbb{B}_{\epsilon}(\ox+\gamma\ov)$, which implies by Theorem~\ref{convexviaMo} that $\vartheta$ is variationally convex at $\ox$ for $\ov$. Utilizing finally Lemma~\ref{cvandstrongcv}, we verify (i) and thus complete the proof.  
\end{proof}

When the modulus of variational strong convexity is not involved, we have yet another equivalence.

\begin{Theorem}[\bf equivalence between variational strong convexity and local strong convexity of Moreau envelopes]\label{stwithoutmodulus} Let $\varphi:\R^n\to \overline{\R}$  be an l.s.c.\ and prox-bounded function with $\bar{x}\in\dom\varphi$, and let $\bar{v}\in\partial\varphi(\bar{x})$.  Then the following assertions are equivalent:

{\bf(i)} $\varphi$ is variationally strongly convex at $\ox$ for $\ov$.

{\bf (ii)} $\varphi$ is prox-regular at $\ox$ for $\ov$ and $e_\lambda \varphi$ is locally strongly convex around $\ox+\lambda \ov$ for all small $\lm>0$. 
\end{Theorem}
\begin{proof} Implication (i)$\Longrightarrow$(ii) is an immediate consequence of Theorem~\ref{stronMore}. Suppose that (ii) holds and find $\lambda_0>0$ such that $e_\lambda\varphi$ is locally strongly convex around $\ox+\lambda\ov$ with {\em some} modulus $\sigma_\lambda>0$ for any $\lambda\in(0,\lambda_0)$. It follows from Proposition~\ref{C11} that there exists a $\varphi$-attentive $\gamma$-localization $T$ of $\varphi$ given by \eqref{att} such that for all $\lambda>0$ sufficiently small there is a convex neighborhood $U_\lambda$ of $\ox+\lambda\ov$ on which $e_\lambda\varphi$ is $\mathcal{C}^{1,1}$ on $U_\lambda$ satisfying ${\rm Prox}_{\lm\ph}(\ox+\lm\ov)=\ox$ with \eqref{GradEnvelope}. Fix $\lambda\in(0,\lambda_0)$ and suppose without loss of generality that $e_\lambda\varphi$ is strongly convex on $U_\lambda$ with modulus $\sigma_\lambda$. Utilizing the first-order description of ${\cal C}^1$-smooth strongly convex functions from \cite[Theorem~4.1.1]{hl04} gives us
\begin{equation}\label{strongconvexMoreau}
e_\lambda\varphi(x)\geq e_\lambda\varphi(u)+\langle \nabla e_\lambda\varphi(u),x-u\rangle+\frac{\sigma_\lambda}{2}\|x-u\|^2\;\mbox{ for all }\;x,u\in U_\lambda. 
\end{equation} 
Let $U$ and $V$ be neighborhoods of $\bar{x}$ and $\bar{v}$, respectively, such that $U\subset{\rm int}\,\mathbb{B}_\gamma(\ox)$, $V\subset{\rm int}\,\mathbb{B}_\gamma (\ov)$, and 
\begin{equation}\label{neighborst}
x+\lambda v\in U_\lambda\;\mbox{ whenever }\;x\in U,\;v\in V. 
\end{equation}
By Lemma~\ref{strongconvex}, it is sufficient to prove that 
\begin{equation}\label{strongconvexineq}
\varphi(x)\ge\varphi(u)+\langle v,x-u\rangle+\frac{\sigma_\lambda}{2}\|x-u\|^2\;\mbox{ for all }\;x\in U,\;(u,v)\in \gph\partial\varphi\cap(U_\gamma\times V),
\end{equation}
where $U_\gamma$ is defined in Lemma~\ref{strongconvex}. Picking $x\in U$ and $(u,v)\in\gph\partial\varphi\cap(U_\gamma\times V)$, we get by \eqref{att} that $(u,v)\in\gph T$, which yields $u+\lambda v\in(\lambda I+T^{-1})(v)$. It follows from \eqref{GradEnvelope} and \eqref{neighborst} that  
$$
\nabla e_\lambda\varphi(u+\lambda v)=(\lambda I+T^{-1})^{-1}(u+\lambda v)=v,
$$
$$
\text{\rm Prox}_{\lambda\varphi}(u+\lambda v)=u+\lambda v-\lambda\nabla e_\lambda\varphi(u+\lambda v)=u
$$
leading us to the envelope representation
\begin{align*}
e_\lambda\varphi(u+\lambda v)& =\varphi\big(\text{\rm Prox}_{\lambda\varphi}(u+\lambda v)\big)+\frac{1}{2\lambda}\big\|u+\lambda v-\text{\rm Prox}_{\lambda\varphi}(u+\lambda v)\big\|^2\\
&=\varphi(u)+\frac{1}{2\lambda}\|u+\lambda v-u\|^2=\varphi(u)+\frac{\lambda}{2}\|v\|^2.
\end{align*}
Apply further the estimate in \eqref{strongconvexMoreau} to get the relationships
\begin{eqnarray*}
\varphi(x)+ \frac{\lambda}{2}\|v\|^2&=&\varphi(x) +\frac{1}{2\lambda}\|x+\lambda v -x\|^2\ge e_\lambda\varphi(x+\lambda v)\\
&\ge& e_\lambda\varphi(u+\lambda v)+\langle\nabla e_\lambda\varphi(u+\lambda v), x+\lambda v-(u+\lambda v)\rangle+\frac{\sigma_\lambda}{2}\|x-u\|^2\\
&=&\varphi(u)+\frac{\lambda}{2}\|v\|^2+\langle v,x-u\rangle+\frac{\sigma_\lambda}{2}\|x-u\|^2.
\end{eqnarray*}
Subtracting the term $\frac{\lambda}{2}\|v\|^2$ from both sides above, we arrive at 
\eqref{strongconvexineq} and thus complete the proof. 
\end{proof}

Note that the characterization of variational convexity and variational strong convexity for extended-real-valued functions, which are obtained in Sections~\ref{sec:Moreau} and \ref{sec:Moreau1} via Moreau envelopes, significantly employ the {\em prox-regularity} of the function in question while not its {\em continuous} prox-regularity, i.e., the {\em subdifferential continuity} of the function is not assumed. In contrast, the latter property is essential for establishing coderivative characterizations of both variational convexity properties that are given in the next section.\vspace*{-0.15in}

\section{Coderivative-Based Characterizations of Variational Convexity}\label{sec:codcharacterizationconvex}\vspace*{-0.05in}
\setcounter{equation}{0}

This section is devoted to deriving complete characterizations of both variational convexity and variational strong convexity properties of {\em continuously prox-regular} functions via their {\em combined} and {\em limiting second-order subdifferentials} taken from Definition~\ref{2nd}.\vspace*{0.05in} 

We begin with the lemma, which provides a second-order {\em sufficient condition} for convexity of smooth functions.

\begin{Lemma}[\bf second-order condition for convexity of $\mathcal{C}^1$-smooth functions]\label{sufconvex} Let $\varphi:\Omega\to\R$ be  a $\mathcal{C}^1$-smooth function, where $\emp\ne\Omega\subset\R^n$ be an open convex set. Then $\varphi$ is convex on $\Omega$ if we have
$$
\langle z,w\rangle\ge 0\quad\text{for all }\;z\in (\widehat{D}^*\nabla\varphi)(x)(w),\;x \in\Omega,\;w\in \R^n. 
$$
\end{Lemma}
\begin{proof} It follows from the proof of \cite[Theorem~3.1]{ChieuHuy11}.
\end{proof}

Now we are ready to obtain {the main coderivative-based} second-order characterizations of {\em variational convexity}.

\begin{Theorem}[\bf second-order subdifferential characterizations of variational convexity]\label{2ndconvexsub} Let $\varphi:\R^n\to\overline{\R}$ be subdifferentially continuous at $\bar{x}\in\dom\varphi$ for $\bar{v}\in\partial\varphi(\bar{x})$. Then the following  assertions  are equivalent:

{\bf(i)} $\varphi$ is variationally convex at $\ox$ for $\ov$. 

{\bf(ii)} $\varphi$ is prox-regular at $\ox$ for $\ov$ and there exist neighborhoods $U$ of $\bar{x}$ and $V$ of $\bar{v}$ such that
\begin{equation}\label{PSDcombinesub}
\langle z,w\rangle\ge 0\;\text{ whenever}\;z\in\breve{\partial}^2\varphi(x,y)(w),\; (x,y)\in\gph\partial\varphi\cap(U\times V),\;w\in\R^n.
\end{equation} 
				
{\bf(iii)} $\varphi$ is prox-regular at $\ox$ for $\ov$ and there exist neighborhoods $U$ of $\bar{x}$, and $V$ of $\bar{v}$ such that
\begin{equation}\label{PSDlimit}
\langle z,w\rangle\ge 0\;\text{ whenever }\;z\in \partial^2\ph(x,y)(w),\;(x,y)\in\gph\partial\varphi\cap(U\times V),\;w\in\R^n.
\end{equation} 
\end{Theorem}
\begin{proof} 
We start with verifying implication (i)$\Longrightarrow$(iii). As mentioned, variational convexity in (i) yields the prox-regularity of $\varphi$ at $\ox$ for $\ov$. Due to the imposed subdifferential continuity of $\varphi$, this ensures the existence of neighborhoods $U$ of $\ox$ and $V$ of $\ov$ as well as of an l.s.c.\ convex function $\psi$ such that $\psi\le\varphi$ on $U$ and
\begin{equation}\label{var=}
(U\times V)\cap\gph\partial\varphi=(U\times V)\cap\gph\partial\psi\quad\text{and }\; \varphi(x)=\psi(x)\;\text{ at the common elements }\;(x,v). 
\end{equation} 
Picking any $(x,y)\in\gph\partial\varphi\cap(U\times V)$, $w\in \R^n$, and $z\in \partial^2\varphi(x,y)(w)$ implies by \eqref{var=} that
$$
\partial^2\varphi(x,y)(w)=\partial^2\psi(x,y)(w),
$$
which tells us in turn that $z\in\partial^2\psi(x,y)(w)$. Applying now to $\psi$ the second-order necessary condition for convexity of l.s.c.\ functions from \cite[Theorem~3.2]{ChieuChuongYaoYen}, we get $\langle z,w\rangle\ge 0$ and therefore justifies assertion (iii). Implication (iii)$\Longrightarrow$(ii) follows from the inclusion in \eqref{2ndinclusioncombinelimit}, valid for all $(x,y)\in\gph\partial\varphi$ and $w\in\R^n$, {which thus verifies \eqref{PSDcombinesub}}.\vspace*{0.03in}   	
			
It remains to verify the reverse implication (ii)$\Longrightarrow$(i). To furnish this, observe first that the imposed prox-regularity in (ii) allows us to suppose without loss of generality that $\ph$ is prox-bounded. Indeed, we can always get this property by adding to $\varphi$ the indicator function of some compact set containing a neighborhood of $\ox$. Therefore, the properties of Moreau envelopes from Proposition~\ref{C11} give us $\lambda_0>0$	such that for any $\lambda\in(0,\lambda_0)$ there exists a convex neighborhood $U_\lambda$ of  $\bar{x}+\lambda\bar{v}$ on which
\begin{equation}\label{propertyprox2nd}
\text{\rm the Moreau envelope $e_\lambda\varphi$ is $\mathcal{C}^{1,1}$ on $U_\lambda$}\;\text{ and }\;\nabla e_\lambda\varphi(x)=\big(\lambda I +(\partial\varphi)^{-1}\big)^{-1}(x)\;\text{ for all }\;x\in U_\lambda. 
\end{equation}
Since (ii) holds, there are neighborhoods $U$ of $\bar{x}$ and $V$ of $\bar{v}$ where \eqref{PSDcombinesub} is satisfied. Fixing any $\lambda\in(0,\lambda_0)$, it follows from the continuity of the mappings $\nabla e_\lambda\varphi$ and $I-\lambda\nabla e_\lambda\varphi$ around $\bar{x}+\lambda\bar{v}$ and from $\nabla e_\lambda\varphi(\ox+\lambda \ov)=\ov$ with $(I-\lambda\nabla e_\lambda \varphi)(\ox+\lambda \ov)=\ox$ that there exists a convex neighborhood $\widetilde{U}\subset U_\lambda$ of $\ox+\lambda\ov$ on which
\begin{equation}\label{Tcont2}
(I-\lambda\nabla e_\lambda\varphi)(x)\times\nabla e_\lambda\varphi(x)\in U\times V,\quad x\in\widetilde{U}. 
\end{equation}
By using Lemma~\ref{sufconvex}, let us now show that $e_\lambda \varphi$ is convex on $\widetilde{U}$. Indeed, take $x\in\widetilde{U},\;u\in\R^n$, and $z~\in~(\widehat{D}^*\nabla e_\lambda\varphi)(x)(u)$. It is easy to check that $-u\in \widehat{D}^*(\nabla e_\lambda\varphi)^{-1}(\nabla e_\lambda\varphi(x),x)(-z)$. Furthermore, it follows from \eqref{propertyprox2nd} and the coderivative sum rule in \cite[Theorem~1.62]{Mordukhovich06} that
\begin{eqnarray*}
\widehat{D}^*(\nabla e_\lambda\varphi)^{-1}\big(\nabla e_\lambda \varphi(x),x\big)(-z)&=&\widehat{D}^*\big(\lambda I+(\partial\varphi)^{-1}\big)\big(\nabla e_\lambda\varphi(x),x\big)(-z)\\
&=&-\lambda z+\widehat{D}^*(\partial\varphi)^{-1}\big(\nabla e_\lambda \varphi(x),x-\lambda\nabla e_\lambda\varphi(x)\big)(-z),
\end{eqnarray*}
which implies that $\lambda z-u\in\widehat{D}^*(\partial\varphi)^{-1}(\nabla e_\lambda\varphi(x),x-\lambda\nabla e_\lambda\varphi(x))(-z)$. In other words,
$$
z\in(\widehat{D}^*\partial\varphi)\big(x-\lambda\nabla e_\lambda\varphi(x),\nabla e_\lambda\varphi(x)\big)(u-\lambda z). 
$$
Using \eqref{PSDcombinesub} and \eqref{Tcont2}, we obtain the inequality
$\langle z,u-\lambda z\rangle\ge 0$, and so $\langle z,u\rangle\ge 0$. Therefore,
\begin{equation*}
\langle z,u\rangle\ge 0\;\mbox{ for all }\;z\in(\widehat{D}^*\nabla e_\lambda\varphi)(x)(u),\;u\in\R^n,\;x\in \widetilde{U}.
\end{equation*}
By Lemma~\ref{sufconvex}, the function $e_\lambda\varphi$ is convex on $\widetilde{U}$. Finally, we apply Theorem~\ref{convexviaMo} on the characterization of variational convexity via Moreau envelopes, which verifies (i) and thus completes the proof of the theorem.
\end{proof} 

As a consequence of Theorem~\ref{2ndconvexsub}, we arrive at the new second-order sufficient conditions for minimizing extended-real-valued continuously prox-regular functions. 

\begin{Corollary}[\bf second-order sufficient optimality conditions for continuously prox-regular functions] Let $\varphi:\R^n\to\overline{\R}$ be continuously prox-regular at $\bar{x}\in\dom\varphi$ for $0\in\partial\varphi(\bar{x})$. Then $\bar{x}$ is a local minimizer of $\varphi$ if either one of the following conditions is satisfied:

{\bf(i)} There exist neighborhoods $U$ of $\bar{x}$ and $V$ of the origin in $\R^n$ such that
$$
\langle z,w\rangle\ge 0\;\text{ whenever }\;z\in\breve{\partial}^2\varphi(x,y)(w),\; (x,y)\in\gph\partial\varphi\cap(U\times V),\;w\in\R^n.
$$
				
{\bf(ii)} There are neighborhoods $U$ of $\bar{x}$ and $V$ of the origin in $\R^n$ such that
$$
\langle z,w\rangle\ge 0\;\text{ whenever }\;z\in\partial^2\varphi(x,y)(w),\; (x,y)\in\gph\partial\varphi\cap(U\times V),\;w\in\R^n.
$$
\end{Corollary}
\begin{proof} It follows immediately from Theorem~\ref{2ndconvexsub} and Remark~\ref{sttilt}(i).
\end{proof}
		
The next theorem provides second-order subdifferential characterizations of {\em variational strong convexity} with a prescribed modulus for extended-real-valued continuously prox-regular functions. 
		
\begin{Theorem}[\bf second-order characterizations of variational strong convexity]\label{2ndstrongconvex} Let $\varphi:\R^n\to\overline{\R}$ be subdifferentially continuous at $\bar{x}\in\dom\varphi$ for $\bar{v}\in\partial\varphi(\bar{x})$. Then the following assertions are equivalent:

{\bf(i)} $\varphi$ is variationally strongly convex at $\ox$ for $\ov$ with modulus $\sigma>0$. 

{\bf(ii)} $\varphi$ is prox-regular at $\ox$ for $\ov$ and there exist neighborhoods $U$ of $\bar{x}$ and $V$ of $\bar{v}$ such that
\begin{equation}\label{PSDcombine2}
\langle z,w\rangle\ge\sigma\|w\|^2\;\text{ whenever }\;z\in  \breve{\partial}^2\varphi(x,y)(w),\;(x,y)\in\gph\partial\varphi\cap(U\times V),\;w\in\R^n.
\end{equation} 
				
{\bf(iii)} $\varphi$ is prox-regular at $\ox$ for $\ov$ and there are neighborhoods $U$ of $\bar{x}$ and $V$ of $\bar{v}$ such that
\begin{equation}\label{PSDlimiting2}
\langle z,w\rangle\ge\sigma\|w\|^2\;\text{ whenever }\;z\in  \partial^2\varphi(x,y)(w),\;(x,y)\in\gph\partial\varphi\cap(U\times V),\;w\in\R^n.
\end{equation} 
\end{Theorem}
\begin{proof} First we justify implication (i)$\Longrightarrow$(iii). Since the prox-regularity is automatic under (i), we proceed with verifying \eqref{PSDlimiting2}. It follows from Lemma~\ref{cvandstrongcv} that the $\sigma$-quadratic shift $\vt:=\varphi-\frac{\sigma}{2}\|\cdot-\ox\|^2$ is variationally convex at $\ox$ for $\ov$. Applying to $\vt$ the second-order characterization of variational convexity from Theorem~\ref{2ndconvexsub}(iii), we find neighborhoods $\widetilde{U}$ of $\bar{x}$ and $\widetilde{V}$ of $\bar{v}$ such that
\begin{equation}\label{PSDstrong}
\langle z, w\rangle\ge 0\;\text{ whenever }\;z \in  \partial^2\vt(x,y)(w),\;(x,y)\in\gph\partial\vt\cap(\widetilde{U}\times\widetilde{V}),\;w\in\R^n.
\end{equation} 
The elementary sum rule for limiting subgradients from \cite[Proposition~1.30]{Mor18} yields
\begin{equation}\label{sumrule1st2}
\partial\vt(x)=\partial\varphi(x)-\sigma(x-\bar{x})\;\text{ for any }\;x\in\dom\varphi.
\end{equation}	
Defining the affine transformation $L:\R^n\times\R^n\to\R^n\times\R^n$ by $L(x,y):=(x,y-\sigma(x-\ox))$ for all $(x,y)\in\R^n\times\R^n$, we find neighborhoods $U$ of $\ox$ and $V$ of $\ov$ such that $L(U\times V)\subset\widetilde{U}\times\widetilde{V}$, which implies that 
\begin{equation}\label{inclusionsubgra}
(x,y)\in U\times V\Longrightarrow(x,y-\sigma(x-\ox) )\in\widetilde{U}\times\widetilde{V}. 
\end{equation} 
To verify now \eqref{PSDlimiting2}, pick any $(x,y)\in\gph\partial\varphi\cap(U\times V)$ and deduce from \eqref{sumrule1st2} and \eqref{inclusionsubgra} that
\begin{equation}\label{inclusiongra2}
\big(x,y-\sigma(x-\ox)\big)\in\gph\partial\vt\cap(\widetilde{U}\times\widetilde{V}).
\end{equation} 
Fix $w\in\R^n$ and $z\in\partial^2\varphi(x,y)(w)=({D}^*\partial\varphi)(x,y)(w)$. It follows from the limiting coderivative sum rule in \cite[Theorem~1.62(ii)]{Mordukhovich06} that we have
$$
(D^*\partial\varphi)(x,y)(w)= D^*\partial\vt\big(x,y-\sigma(x-\bar{x})\big)(w)+\sigma w,
$$
which tells us that $z-\sigma w\in(D^*\partial\vt)(x,y-\sigma(x-\bar{x}))(w)=\partial^2\vt(x,y-\sigma(x-\bar{x}))(w)$. 
Combining the latter inclusion with \eqref{PSDstrong} and \eqref{inclusiongra2}, we arrive at $\langle z-\sigma w,w\rangle\ge 0$, i.e., at $\langle z,w\rangle\ge\sigma\|w\|^2$, which gives us \eqref{PSDlimiting2} and thus verifies the claimed implication  (i) $\Longrightarrow$(iii). The one in (iii)$\Longrightarrow$(ii) follows from the inclusion in \eqref{2ndinclusioncombinelimit}.\vspace*{0.03in}

To complete the proof of the theorem, it remains to justify implication (ii)$\Longrightarrow$(i). Suppose that $\varphi$ is prox-regular at $\ox$ for $\ov$ and then find neighborhoods $U$ of $\bar{x}$ and $V$ of $\bar{v}$ such that \eqref{PSDcombine2} is satisfied. Our immediate goal is to check that the function $\vt:=\varphi-\frac{\sigma}{2}\|\cdot-\ox\|^2$ is variationally convex at $\ox$ for $\ov$. To furnish this, consider another affine transformation $L_1(x,y):=(x,y+\sigma (x-\ox))$ on $\R^n\times\R^n$ for which there are neighborhoods $\widetilde{U}$ of $\ox$ and $\widetilde{V}$ of $\ov$ with $L_1(\widetilde{U}\times\widetilde{V})\subset U\times V$. It follows from \eqref{sumrule1st2} that
\begin{equation}\label{inclusionsubgra2}
(x,y)\in\gph\partial\vt\cap(\widetilde{U}\times\widetilde{V})\Longrightarrow\big(x,y+\sigma(x-\ox)\big)\in\gph\partial\varphi\cap(U\times V).  
\end{equation}
Take $(x,y)\in\gph\partial\vt\cap(\widetilde{U}\times\widetilde{V})$, $w\in\R^n$, and $z\in(\widehat{D}^*\partial\vt)(x,y)(w)$. By using the regular coderivative sum rule from \cite[Theorem~1.62(i)]{Mordukhovich06}, we get
$$
(\widehat{D}^*\partial\vt)(x,y)(w)=\widehat{D}^*\partial\varphi\big(x,y+\sigma(x-\bar{x})\big)(w)-\sigma w,
$$
which ensures that $z+\sigma w\in(\widehat{D}^*\partial\varphi)(x,y+\sigma(x-\bar{x}))(w)$. It follows from \eqref{PSDcombine2} and \eqref{inclusionsubgra2} that $\langle z+\sigma w,w\rangle \ge\sigma\|w\|^2$, i.e., $\langle z,w\rangle\ge 0$. Applying the second-order subdifferential characterization of variational convexity from Theorem~\ref{2ndconvexsub}(ii) tells us that $\vt$ is variationally convex at $\ox$ for $\ov$. To deduce from here the strong variational convexity of $\ph$ at $\ox$ for $\ov$, we just employ Lemma~\ref{cvandstrongcv} and therefore complete the proof of the theorem.
\end{proof}

Note that the obtained characterizations of variational convexity and variational strong convexity expressed in terms of the {\em limiting second-order subdifferential} \eqref{limitsec} are {\em more preferable} in comparison with their counterparts given via the combined second-order subdifferential \eqref{seccombine}. This is due the {\em robustness} and  well-developed {\em calculus rules} for the former construction that is not the case for the latter one. Nevertheless, the characterizations derived in Theorems~\ref{2ndconvexsub} and \ref{2ndstrongconvex} even in terms of \eqref{limitsec} involve {\em neighborhoods} of the reference point. It is definitely desired to establish {\em pointbased} characterizations of these properties expressed via \eqref{limitsec} just at the point in question.\vspace*{0.05in}

We now obtain such a pointbased characterization for the variational {\em strong} convexity property of continuously prox-regular functions. Establishing results of this type for (nonstrong) variational convexity is an {\em open question}.\vspace*{0.05in}

To proceed, recall the following lemma taken from \cite[Proposition~4.6]{ChieuLee17}.

\begin{Lemma}[\bf pointbased second-order sufficient condition for local strong convexity of functions from class $\mathcal{C}^{1,1}$]\label{sufconvex3} Let   $\varphi:\R^n\to\R$ be a function of class $\mathcal{C}^{1,1}$ around $\ox\in\R^n$. Then $\varphi$ is locally strongly convex around $\ox$ if we have the positive-definiteness condition
$$
\langle z,w\rangle>0\;\text{ for all }\;z\in({D}^*\nabla\varphi)(\ox)(w),\;w\in \R^n\setminus\{0\}. 
$$
\end{Lemma}

Here is a pointbased characterization of variational strong convexity for
subdifferentially continuous functions.

\begin{Theorem}[\bf pointbased second-order subdifferential characterization of variational strong convexity]\label{pointPD} Let $\varphi:\R^n\to\overline{\R}$ be a subdifferentially continuous function at $\bar{x}\in\dom\varphi$ for $\bar{v}\in\partial\varphi(\bar{x})$. Then the following assertions are equivalent:

{\bf(i)} $\varphi$ is variationally strongly convex at $\ox$ for $\ov$.

{\bf(ii)} $\varphi$ is prox-regular at $\ox$ for $\ov$ and the second-order subdifferential \eqref{limitsec} is positive-definite in the sense that
\begin{equation}\label{PSDpointbased}
\langle z,w\rangle>0\;\text{ whenever }\;z\in\partial^2\varphi(\ox,\ov)(w),\;w\ne 0.
\end{equation} 
\end{Theorem}
\begin{proof} Implication (i)$\Longrightarrow$(ii) follows from Theorem~\ref{2ndstrongconvex}. To verify the converse implication (ii)$\Longrightarrow$(i), observe as above that the prox-regularity of $\ph$ imposed in (ii) ensures without loss of generality that $\ph$ is prox-bounded. By the properties of Moreau envelopes listed in Proposition~\ref{C11}, we find $\lambda_0>0$ such that for any $\lambda\in(0,\lambda_0)$ there exists a convex neighborhood $U_\lambda$ of  $\bar{x}+\lambda\bar{v}$ on which
\begin{equation*}
\text{\rm the Moreau envelope $e_\lambda\varphi$ is $\mathcal{C}^{1,1}$ on $U_\lambda$}\quad\text{and }\;\nabla e_\lambda\varphi(x)=\big(\lambda I +(\partial\varphi)^{-1}\big)^{-1}(x)\quad\text{for all }\;x\in U_\lambda. 
\end{equation*}
Fix  any $\lambda\in(0,\lambda_0)$ and show that $e_\lambda \varphi$ is locally strongly convex around $\ox +\lambda\ov$. Indeed, take $z\in({D}^*\nabla e_\lambda\varphi)(\ox+\lambda\ov)(u)$ with $u\ne 0$. It follows from \cite[Lemma~6.4]{BorisKhanhPhat} that $z\in \partial^2\varphi(\ox,\ov)(u-\lambda z)$. If $u-\lambda z \ne 0$, then we get from \eqref{PSDpointbased} that
$\langle z,u-\lambda z\rangle>0$, i.e., $\langle z,u\rangle>\lambda\|z\|^2>0$. Otherwise, this tells us that $u=\lambda z$, and thus 
$$
\langle z,u\rangle=\frac{1}{\lambda}\langle u,u\rangle=\frac{1}{\lambda}\|u\|^2>0.
$$
Therefore, we always have the coderivative positive-definiteness for the ${\cal C}^{1,1}$ Moreau envelope:
\begin{equation*}
\langle z,u\rangle>0\;\mbox{ for all }\;z\in({D}^*\nabla e_\lambda\varphi)(\ox+\lambda\ov)(u),\;u\ne 0.
\end{equation*}
This ensures by Lemma~\ref{sufconvex3} that $e_\lambda\varphi$ is locally strongly convex around $\ox+\lambda\ov$. Employing the characterization of Theorem~\ref{stwithoutmodulus} concludes that $\varphi$ is variationally strongly convex at $\ox$ for $\ov$, and thus we complete the proof. 
\end{proof}

As a consequence of the obtained results, we arrive to the following second-order subdifferential characterizations of tilt-stable minimizers, with a given modulus, for extended-real-valued  continuously prox-regular functions.

\begin{Corollary}[\bf second-order characterizations of tilt stability for continuously prox-regular functions]\label{2ndtilt} Let $\varphi:\R^n\to \overline{\R}$ be continuously prox-regular at $\bar{x}\in\dom\varphi$ for $0\in\partial\varphi(\bar{x})$. Then we have the equivalences: 

{\bf(i)} $\bar{x}$ is a tilt-stable local minimizer with modulus $\kappa>0$. 

{\bf(ii)} There exist neighborhoods $U$ of $\bar{x}$ and $V$ of the origin in $\R^n$ such that
$$
\langle z,w\rangle\ge\frac{1}{\kappa}\|w\|^2\;\text{ whenever }\;z\in \breve{\partial}^2\varphi(x,y)(w),\;(x,y)\in\gph\partial\varphi\cap(U\times V),\;w\in\R^n.
$$ 
				
{\bf(iii)} There exist neighborhoods $U$ of $\bar{x}$ and $V$ of the origin in $\R^n$ such that
$$
\langle z,w\rangle\ge\frac{1}{\kappa}\|w\|^2\;\text{ whenever }\;z\in  \partial^2\varphi(x,y)(w),\;(x,y)\in\gph\partial\varphi\cap(U\times V),\;w\in\R^n.
$$
Furthermore, the fulfillment of {\rm(i)} with some modulus is equivalent to the pointbased positive-definiteness condition 
\begin{equation}\label{tilt-point}
\langle z,w\rangle>0\;\text{ whenever }\;z\in\partial^2\varphi(\ox,0)(w),\;w\ne 0.
\end{equation}
\end{Corollary}
\begin{proof}  The given characterizations follow directly from {Theorems~\ref{2ndstrongconvex}, \ref{pointPD} and Proposition~\ref{equitiltstr}}.
\end{proof}

\begin{Remark}[\bf discussions on characterizing tilt-stability via second-order subdifferentials] {\rm The first characterization of tilt-stable minimizers of continuously prox-regular functions on $\R^n$ was obtained in the pointwise second-order subdifferential form \eqref{tilt-point} by Poliquin and Rockafellar \cite[Theorem~1.3]{Poli} without involving moduli. The quantitative neighborhood characterization of tilt stability in (ii) via the combined second-order subdifferential was established by Mordukhovich and Nghia in \cite[Theorem~4.3]{MorduNghia13} in Hilbert spaces. {Here we recover the known results including the equivalence between (i) and (ii) and the pointbased characterization of tilt stability by using the obtained new characterizations of variational convexity.} To the best of our knowledge, the quantitative characterizations in (iii) via the limiting second-order subdifferential has not been observed earlier in the literature. Note that our proof of the second-order subdifferential characterizations of tilt-stable minimizers in Corollary~\ref{2ndtilt} via variational convexity is significantly different from those previously used.}
\end{Remark}

 We conclude this section with the following remark on the novel SCD approach to tilt stability and related issues developed of Gfrerer and Outrata in their recent paper \cite{go22}.

\begin{Remark}[\bf strong variational convexity and tilt stability via SCD mappings]\label{discussSCD} \rm In \cite{go22}, Gfrerer and Outrata introduced new generalized derivatives for set-valued mappings called \textit{SCD mappings} (subspace containing derivative mappings). They derived for such mappings some calculus rules and revealed a number of interesting connections in variational analysis. SCD mappings particularly contain subdifferentials of prox-regular and subdifferentially continuous functions \cite[Proposition 3.26]{go22}. In particular, the SCD connection to Rockafellar’s quadratic bundles \eqref{qb} was established and a comparison with the second-order subdifferentials was provided in \cite[Example~3.29]{go22}. Characterizations of tilt stability and strong metric regularity were also derived in Section~7 of \cite{go22}. The equivalent descriptions of variational strong convexity given in Proposition~\ref{equitiltstr} allow us to obtain the corresponding characterizations of variational strong convexity via SCD mappings.
\end{Remark}\vspace*{-0.22in}

\section{Variational Sufficiency in Composite
Optimization}\label{sec:varsuf}\vspace*{-0.05in}
\setcounter{equation}{0}

 {As discussed in Remark~\ref{sttilt}}, variational convexity and variational strong convexity of an l.s.c.\ function $\varphi\colon\R^n\to\oR$ imply the local optimality and  {tilt-stability of a local minimizer}, respectively. However, the reverse implications fail. The properties of variational convexity and variational strong convexity of minimizing cost functions at stationary points were labeled by Rockafellar \cite{r19} as {\em variational sufficiency} and {\em strong variational sufficiency}, where the motivations came from applications to proximal point algorithms. Further applications of these notions and ideas to extended augmented Lagrangian methods in rather general frameworks of optimization were given in the most recent papers \cite{roc,r22}.\vspace*{0.03in}

The main goal of this section is to investigate and characterize both of these properties in structured frameworks of {\em composite constrained optimization} by using the {\em generalized differential characterizations} of variational convexity and variational strong convexity established in Sections~\ref{sec:Moreau} and \ref{sec:Moreau1} being married to {\em second-order calculus rules} developed for the basic second-order subdifferential \eqref{limitsec}.\vspace*{0.05in}

We start with the definitions of variational sufficiency and strong variational sufficiency in the general unconstrained format of minimizing extended-real-valued functions.

\begin{Definition}[\bf variational sufficiency and strong variational sufficiency for local optimality]\label{vrforun} {\rm Given an extended-real-valued function $\varphi:\R^n\to\overline{\R}$, consider the unconstrained optimization problem:
\begin{equation}\label{unconpb}
\min\quad\varphi(x)\quad\text{subject to }\;x\in\R^n.
\end{equation}
It is said that the {\em variational sufficient condition for local optimality} in \eqref{unconpb} holds at $\ox$ if $\varphi$ is variationally convex at $\ox$ for $0\in \partial\varphi(\ox)$. If $\varphi$ is variationally strongly convex at $\ox$ for $0$ with modulus $\sigma>0$, then we say that the {\em strong variational sufficient condition for local optimality at $\bar{x}$} holds with modulus $\sigma$.}
\end{Definition}

Here we consider the class of {\em composite optimization problems} given by:
\begin{equation}\label{composite}
\min\;\varphi(x):=\ph_0(x)+\psi\big(g(x)\big)\quad\text{subject to }\;x\in\R^n,
\end{equation}
where $\psi:\R^m\to\overline{\R}$ is an extended-real-valued l.s.c.\ function, $\ph_0:\R^n\to\R$ is a $\mathcal{C}^2$-smooth function, and $g$ is a $\mathcal{C}^2$-smooth mapping from $\R^n$ to $\R^m$. These are our {\em standing assumptions} in this section unless otherwise stated. As mentioned, the composite format \eqref{composite} implicitly incorporates the constraint $g(x)\in\dom\psi$. In what follows, we aim at deriving second-order characterizations of variational sufficiency and strong variational sufficiency for local optimality in \eqref{composite} expressed in terms of the given data of this problem.\vspace*{0.03in}

To proceed, for each $(x,v)\in\R^n\times\R^n$ define the {\em set of multipliers} \begin{equation}\label{Lagrangemultiplier}
\Lambda (x,v):=\big\{y\in\R^m\;\big|\;v=\nabla\ph_0(x)+\nabla g(x)^*y,\;y\in\partial\psi\big(g(x)\big)\big\}.
\end{equation}

The first theorem imposes the {\em full rank} assumption on the Jacobian matrix $\nabla g(\ox)$ at the point in question.

\begin{Theorem}[\bf variational and strong variational sufficiency for composite problems with full ranks of Jacobians]\label{vrcoderivative}  Let $\ox\in\R^n$ be a stationary point of the composite optimization problem \eqref{composite} at which  $\rank\nabla g(\ox)=m$ and hence there exists a unique vector $\oy\in\R^m$ with
\begin{equation}\label{KKTcomposite}
\nabla\ph_0(\ox)+\nabla g(\ox)^*\oy = 0\;\mbox{ and }\;\oy\in\partial\psi\big(g(\ox)\big).
\end{equation}
 Suppose in addition that $\psi$ is subdifferentially continuous at $g(\ox)$ for $\oy$. Then we have the following assertions:

{\bf(i)} The variational sufficient condition for local optimality in \eqref{composite} holds at $\ox$ if and only if $\psi$ is prox-regular at $g(\ox)$ for $\oy$ and there exist neighborhoods $U$ of $\ox$ and $V$ of $0$ such that 
\begin{equation}\label{cvsufficiency}
\langle\nabla^2\ph_0(x)w,w\rangle+\langle\nabla^2\langle y,g\rangle(x)w,w\rangle+\langle u,\nabla g(x)w\rangle\ge 0
\end{equation}
for all $x\in U$, $v\in V$, $y\in\Lambda(x,v)$, $u\in\partial^2\psi(g(x),y)(\nabla g(x)w)$, $w\in\R^n$, where $\Lambda(x,v)$ is a singleton in this case.

{\bf(ii)} The strong variational sufficient condition for local optimality in \eqref{composite} holds at $\ox$ with modulus $\sigma>0$ if and only if  $\psi$ is  prox-regular at $g(\ox)$ for $\oy$ and there exist neighborhoods $U$ of $\ox$ and $V$ of $0$ such that 
\begin{equation}\label{strsufficiency}
\langle\nabla^2\ph_0(x)w,w\rangle+\langle\nabla^2\langle y,g\rangle(x)w,w\rangle+\langle u,\nabla g(x)w\rangle\ge\sigma\|w\|^2
\end{equation}
for all $x\in U$, $v\in V$, $y\in\Lambda(x,v)$, $u\in\partial^2\psi(g(x),y)(\nabla g(x)w)$, $w\in\R^n$, where $\Lambda(x,v)$ is a singleton in this case.

Furthermore, the strong variational sufficiency in {\rm(ii)} with some modulus $\sigma>0$ is equivalent to the prox-regularity of $\psi$ at $g(\ox)$ for $\oy$ together with the fulfillment of the pointbased condition
\begin{equation}\label{varsufficiency}
\langle\nabla^2\ph_0(\ox)w,w\rangle+\langle\nabla^2\langle\oy,g\rangle(\ox)w,w\rangle +\langle u,\nabla g(\ox)w\rangle>0
\end{equation}
whenever $u\in\partial^2\psi\big(g(\ox),\oy\big)\big(\nabla g(\ox)w\big)$ and $w\ne 0$.
\end{Theorem}
\begin{proof} Note that the full rank assumption $\rank\nabla g(\ox)=m$ implies that there exists $\epsilon>0$ such that $\rank\nabla g(x)=m$ for all $x\in \mathbb{B}_\epsilon(\ox)$. Thus it follows from the first-order subdifferential sum and chain rules in \cite[Proposition~1.10 and Exercise~1.72]{Mor18} that we have the equivalences
\begin{equation}\label{chainrulecom} 
v\in\partial\varphi(x)\iff v\in\nabla\ph_0(x)+\nabla g(x)^*\partial\psi\big(g(x)\big) \iff\Lambda (x,v)\;\text{ is a singleton}
\end{equation} 
for all $x\in\mathbb{B}_\epsilon(\ox)$, which verifies the existence of a unique $\oy\in\R^m$ satisfying \eqref{KKTcomposite}.

Next we observe that it follows from \eqref{chainrulecom} and the definitions of prox-regularity and subdifferential continuity in Section~\ref{sec:prelim} that the prox-regularity and subdifferential continuity properties of the function $\ph$ an \eqref{composite} at $\ox$ for $0$ are equivalent to the corresponding properties of $\psi$ at $g(\ox)$ for $\oy$. Furthermore, Remark~\ref{disvar}(iii) tells us that the variational convexity of a function yields its prox-regularity at the reference point. Therefore, we only need to verify that the variational convexity as well as the variational strong convexity of $\varphi$ at $\ox$ for $0$ with and without modulus $\sigma>0$ hold if and only if the corresponding conditions \eqref{cvsufficiency}, \eqref{strsufficiency}, and \eqref{varsufficiency} are satisfied.

To prove all of this, we employ the second-order subdifferential sum rule from \cite[Proposition~1.121]{Mordukhovich06} and the second-order subdifferential chain rule from \cite[Theorem~1.127]{Mordukhovich06} to {the function $\ph$ in} \eqref{composite}, which tell us that 
\begin{align*}
\partial^2\varphi(x,v)(w)&=\nabla^2\ph_0(x)w+\partial^2(\psi\circ g)\big(x,v-\nabla \ph_0(x)\big)(w)\\
&=\nabla^2\ph_0(x)w+\nabla^2\langle y,g\rangle(x)w+\nabla g(x)^*\partial^2 \psi\big(g(x),y\big)\big(\nabla g(x)w\big)
\end{align*}
for each $x\in\mathbb{B}_\epsilon(\ox)$, $v\in\partial\varphi(x)$, $w\in\R^n$, and the unique vector $y$ from $\Lambda(x,v)$. By the obtained representation of $\partial^2\ph(x,v)$, it follows that the conditions in \eqref{cvsufficiency} and \eqref{strsufficiency} are equivalent to the existence of neighborhoods $U$ of $\ox$ and $V$ of $0$ such that we have the estimates 
\begin{equation}\label{PSDvarphi1} 
\langle z,w\rangle\ge 0\;\text{ for all }\;z\in\partial^2\varphi(x,v)(w),\; (x,v)\in\gph\partial\varphi\cap(U\times V),\;w\in\R^n,
\end{equation} 
\begin{equation}\label{PSDvarphi2}
\langle z,w\rangle\ge\sigma\|w\|^2\;\text{ for all }\;z\in\partial^2\varphi(x,v)(w),\;(x,v)\in\gph\partial\varphi\cap(U\times V),\;w\in\R^n,
\end{equation} 
respectively. In the same way, \eqref{varsufficiency} reduces to the positive-definiteness condition
\begin{equation}\label{PSDvarphi3}
\langle z,w\rangle>0\;\text{ for all }\;z\in\partial^2\varphi(\ox,0)(w),\;w\ne 0.
\end{equation}
Finally, we deduce from the second-order subdifferential characterization in Theorems~\ref{2ndconvexsub}, \ref{2ndstrongconvex}, and \ref{pointPD} that the conditions in \eqref{PSDvarphi1}, \eqref{PSDvarphi2}, and \eqref{PSDvarphi3} are equivalent to the variational convexity of $\varphi$ and the variational strong convexity of $\varphi$ at $\ox$ for $0$ with and without modulus $\sigma$, respectively. This completes the proof of the theorem.
\end{proof}

Our further goals are deriving effective conditions ensuring the fulfillment of variational sufficiency and strong variational sufficiency for local optimality in the composite model \eqref{composite} and also characterizing these properties 
{\em without} imposing the full rank assumption on the Jacobian $\nabla g(\ox)$. To proceed in this direction requires dealing with particular classes of compositions in \eqref{composite}, which cover a large territory in variational analysis and optimization.\vspace*{0.03in}

Following \cite{Rockafellar98}, recall that an l.s.c.\ function $\th:\R^n\to \overline{\R}$ is {\em strongly amenable} at $\ox$ if there exists neighborhood $U$ of $\ox$ on which $\th$ can be represented in the composition form
$\th =\psi\circ g$ with a $\mathcal{C}^2$-smooth mapping $g\colon U\to\R^m$ and a proper l.s.c.\ convex function $\psi:\R^m\to\overline{\R}$ such that the following {\em first-order qualification condition} holds:
\begin{equation}\label{1stqualify}
\partial^\infty\psi(\oz)\cap\ker\nabla g(\ox)^*=\{0\}\;\mbox{ with }\;\oz:=g(\ox)
\end{equation}
If in addition $\psi$ is piecewise linear-quadratic as in \cite{Rockafellar98}, then $\th$ is called {\em fully amenable} at $\ox$. Note that \eqref{1stqualify} is automatically satisfied if either $\psi$ is locally Lipschitzian around $\oz$, or $\nabla g(\ox)$ is of full rank, while neither of these conditions is required for the fulfillment of \eqref{1stqualify}. 

We are also going to use the {\em second-order qualification condition} (SOQC) from \cite{mr} for problem \eqref{composite} at $\ox$, which is formulated as follows:
\begin{equation}\label{SOQC}
\partial^2\psi(\oz,\oy)(0)\cap\ker\nabla g(\ox)^*=\{0\}\;\text{ with }\;\oy\in \partial\psi\big(\oz\big)\;\text{ and }\;\oz:=g(\ox).
\end{equation}

If $\psi$ is {\em convex}, then the following easy relationship between the qualification conditions \eqref{1stqualify} and \eqref{SOQC} holds.

\begin{Proposition}[\bf relationships between the first- and second-order qualification conditions]\label{2ndqualifyimplies1st} Let $\psi\colon\R^m \to\overline{\R}$ be a convex function. Then for each $\bar{z}\in\R^m$ and  $\bar{y}\in\partial\psi(\bar{z})$, we have the inclusions:
\begin{equation}\label{singular2ndinclusion} 
\partial^\infty\psi(\bar{z})\subset\breve{\partial}^2\psi(\bar{z},\bar{y})(0)\subset {\partial}^2\psi(\bar{z},\bar{y})(0). 
\end{equation}
Consequently, the second-order qualification condition \eqref{SOQC} yields the first-order qualification condition \eqref{1stqualify}.
\end{Proposition}
\begin{proof} Pick $v\in\partial^\infty\psi(z)$ and deduced from the convexity of $\psi$ and \cite[Proposition~8.12]{Rockafellar98} that
$$
\langle v,z-\bar{z}\rangle\le 0\quad\text{for all }\;z \in\dom\psi,
$$
which immediately implies that 
$$
\underset{(z,y)\overset{\text{\rm gph}\,\partial\psi}{\to}(\bar{z},\oy)}{\limsup}  \frac{\langle(v,0),(z,y)-(\bar{z},\bar{y})\rangle}{\|z-\bar{z}\|+\|y-\oy\|}\le 0,
$$
and thus $(v,0)\in\widehat{N}_{\text{\rm gph}\,\partial\psi}(\bar{z},\oy)$, i.e., $v\in\breve{\partial}^2\psi(\bar{z},\oy)(0)$. This  justifies the first inclusion in \eqref{singular2ndinclusion}. The second inclusion therein is obvious, and so \eqref{singular2ndinclusion} is verified. Implication \eqref{SOQC}$\Longrightarrow$\eqref{1stqualify} clearly follows from \eqref{singular2ndinclusion}. 
\end{proof}

Next we verify the important {\em robustness} property of SOQC for strongly amenable compositions.

\begin{Lemma}[\bf robustness of the second-order qualification condition]\label{robustSOQClm} Let $\ox\in\R^n$ be a stationary point of the composite optimization problem \eqref{composite} under the standing assumption on $\ph_0,\psi$, and $g$. Suppose in addition that $\psi$ and $g$ be mappings from the composite representation of a strongly amenable function at $\ox$ and that the second-order qualification condition \eqref{SOQC} is satisfied at $\ox$. Then there are neighborhoods $\mathcal{X}$ of $\ox$ and $\mathcal{O}$ of $0$ with 
\begin{equation}\label{robustSOQC}
\partial^2\psi\big(g(x),y\big)(0)\cap\ker\nabla g(x)^*=\{0\}\;\text{ whenever }\;x\in \mathcal{X},\;y\in\partial\psi\big(g(x)\big),\;\mbox{ and }\;\nabla\varphi_0(x)+\nabla g(x)^*y\in\mathcal{O}. 
\end{equation}
\end{Lemma}
\begin{proof} Suppose that \eqref{robustSOQC} fails and then find sequences $x_k\to\ox$,  $y_k\in\partial\psi(g(x_k))$ with
\begin{equation}\label{lim=0}
\lim_{k\to\infty}\big[\nabla\varphi_0(x_k)+\nabla g(x_k)^*y_k\big]=0,   
\end{equation}
\begin{equation}\label{zkincap}
0\ne z_k\in\partial^2\psi\big(g(x^k),y^k\big)(0)\cap\ker\nabla g(x^k)^*\;\text{for all }\ k \in\\N.
\end{equation}
If $\{y_k\}$ bounded, suppose without loss of generality that $y_k\to\bar{y}$ as $k\to\infty$. Put $\tilde{z_k}:=z_k/\|z_k\|$ for all $k\in\\N$. Then $\|\tilde{z_k}\|=1$, and we get without loss of generality that $\tilde{z_k}\to\tilde{z}$ as $k\to\infty$ with $\|\tilde{z}\|=1$. It follows from \eqref{zkincap} that
$$
\tilde{z_k}=\disp\frac{z_k}{\|z_k\|}\in\partial^2\psi\big(g(x^k),y^k\big)(0)\cap\ker\nabla g(x^k)^*\;\text{ for all }\;k\in\\N,
$$
and hence $ \tilde{z}\in\partial^2\psi(g(\ox),\bar{y})(0)\cap\ker\nabla g(\ox)^*$. This yields $\tilde{z}=0$ due to \eqref{SOQC}, a contradiction.

If $\{y_k\}$ is not bounded, we suppose without loss of generality that $\|y_k\|\to \infty$ as $k\to\infty$ and define $\tilde{y_k}:= y_k/\|y_k\|$ for all $k\in\N$. Then $\|\tilde{y_k}\|=1$, and thus get again without loss of generality that $\tilde{y_k} \to \tilde{y}$ as $k\to\infty$ with $\|\tilde{y}\|=1$. It follows from \eqref{lim=0} that $\nabla g(\ox)^*\tilde{y}=0$. To show next that $\tilde{y}\in\partial^\infty\psi(g(\ox))$, we recall that $y_k\in\partial\psi(g(x_k))$ and deduce from the convexity of $\psi$ together with \cite[Proposition~8.12]{Rockafellar98} that 
$$
\psi(z)\ge\psi\big(g(x_k)\big)+\langle y_k,z-g(x_k)\rangle\;\text{ for all }\;z\in\R^m\;\mbox{ and }\;k\in\N, 
$$
which implies in turn the estimate
$$
\langle\tilde{y_k},z -g(x_k)\rangle\le\frac{\psi(z)-\psi\big(g(x_k)\big)}{\|y_k\|}\;\mbox{ as }\;z\in\R^m,\;k\in\N.
$$
Thus we have for each $z\in\dom\psi$ that
\begin{equation}\label{limsupineq}
\langle\tilde{y},z-g(\ox)\rangle=\lim_{k\to\infty}\langle\tilde{y_k},z-g(x_k)\rangle\le\limsup_{k\to\infty}\frac{\psi(z)-\psi\big(g(x_k)\big)}{\|y_k\|}.
\end{equation}
Remembering that $\psi$ is l.s.c.\ leads us to the conditions
\begin{equation}\label{limsupinf}
\limsup_{k\to\infty}\big(-\psi(g(x_k))\big)=-\liminf_{k\to\infty}\psi\big(g(x_k)\big)\le-\psi\big(g(\ox)\big).
\end{equation}
Combining \eqref{limsupineq} and \eqref{limsupinf} with $\psi(z)<\infty$ and $1/\|y_k\|\to 0$ as $k\to\infty$, we conclude that $\langle\tilde{y},z-g(\ox)\rangle\le 0$ for all $z\in\dom\psi$. This tells us that $\tilde{y}\in\partial^\infty\psi(g(\ox))$ and hence $\tilde{y}\in\partial^\infty\psi(g(\ox))\cap\ker\nabla g(\ox)^*$. By the strong amenability of $\psi\circ g$, the first-order qualification condition \eqref{1stqualify} is satisfied, and thus $\tilde{y}=0$. This is a contradiction, which completes the proof of the lemma.
\end{proof} 

Now we are ready to obtain efficient conditions that ensure variational and strong variational sufficiency for composite optimization problems \eqref{composite} without the full rank assumption. These conditions address the general class of {\em strongly amenable} compositions in \eqref{composite}.

\begin{Theorem}[\bf variational and strong variational sufficiency for local optimality with strongly amenable compositions]\label{vrcoderivativenotfullrank} Let $\ox\in\R^n$ be a stationary point of the composite optimization problem \eqref{composite} under the standing assumption on $\ph_0,\psi$, and $g$. Suppose in addition that $\psi$ and $g$ be mappings from the composite representation of a strongly amenable function at $\ox$ and that the second-order qualification condition \eqref{SOQC} is satisfied at $\ox$. Then we have the following assertions:

{\bf(i)} The variational sufficiency for local optimality in \eqref{composite} holds at $\ox$ if there exist neighborhoods $U$ of $\ox$ and $V$ of $0$ such that \eqref{cvsufficiency} is satisfied for all $x\in U$, $v\in V$, $y\in\Lambda(x,v)$, $u\in\partial^2\psi(g(x),y)(\nabla g(x)w)$, and $w\in\R^n$.

{\bf(ii)} The strong variational sufficiency for local optimality in \eqref{composite} holds at $\ox$ with modulus $\sigma>0$ if there exist neighborhoods $U$ of $\ox$ and $V$ of $0$ such that the neighborhood condition \eqref{strsufficiency} is satisfied for all $x\in U$, $v\in V$, $y\in\Lambda(x,v)$, $u\in\partial^2\psi(g(x),y)(\nabla g(x)w)$, and $w\in\R^n$.

{\bf(iii)} The strong variational sufficient condition for local optimality in \eqref{composite} holds at $\ox$ if the pointbased condition \eqref{varsufficiency} is satisfied for any $\oy\in\Lambda(\ox,0)$.
\end{Theorem}
\begin{proof} Note that the first-order qualification condition \eqref{1stqualify} is clearly {\em robust} with respect of small perturbations of the reference point, which means that there exists $\epsilon>0$ such that 
$$
\partial^\infty\psi\big(g(x)\big)\cap\ker\nabla g(x)^*=\{0\}\;\text{ for all}\;x\in \mathbb{B}_\epsilon(\ox).
$$
Thus we can apply the aforementioned subdifferential sum rule and then the chain rule from \cite[Theorem~4.5]{Mor18} to the (lower regular) strongly amenable composition $\psi\circ g$ and get for all $x\in\mathbb{B}_\epsilon(\ox)$ and $w\in\R^n$ the equivalences
\begin{equation*}
v\in\partial\varphi(x)\iff v\in\nabla\ph_0(x)+\nabla g(x)^*\partial\psi\big(g(x)\big)\iff\Lambda(x,v)\ne\emp.
\end{equation*} 
Using Lemma \ref{robustSOQClm}, we find neighborhoods $\mathcal{X}$ of $\ox$ and $\mathcal{O}$ of $0$ such that \eqref{robustSOQC} holds. Since the strong amenability assumption on the composition $\psi\circ g$ yields its continuous prox-regularity at any point $x$ near $\ox$ by \cite[Proposition~13.32]{Rockafellar98}, this property also holds for the cost function $\ph$ in \eqref{composite}. The second-order subdifferential chain rule for strongly amenable functions taken from \cite[Corollary~3.76]{Mordukhovich06} tells us that for each $x$ near $\ox$, each $v\in\partial\varphi(x)\cap \mathcal{O}$, and each $w\in\R^n$ we have
$$
\partial^2(\psi\circ g)\big(x,v-\nabla\ph_0(x)\big)(w)\subset\bigcup_{y\in\Lambda(x,v)}\left[\nabla^2\langle y,g\rangle(x)w+\nabla g(x)^*\partial^2\psi\big(g(x), y\big)\big(\nabla g(x)w\big)\right].    
$$
Combining this inclusion with the second-order subdifferential sum rule from \cite[Proposition~1.121]{Mordukhovich06} gives us
\begin{align}\label{chainrulecomamenable}
\partial^2\varphi(x,v)(w)&=\nabla^2\ph_0(x)w+\partial^2(\psi\circ g)\big(x,v-\nabla \ph_0(x)\big)(w)\nonumber\\
&\subset\bigcup_{y\in\Lambda(x,v)}\left[\nabla^2\ph_0(x)w+\nabla^2\langle y,g\rangle (x)w+\nabla g(x)^*\partial^2\psi\big(g(x),y\big)\big(\nabla g(x)w\big)\right]
\end{align}
for all $w\in\R^n$, $x$ near $\ox$, and $v\in\partial\varphi(x)\cap\mathcal{O}$. It follows from the upper estimates in \eqref{chainrulecomamenable} that conditions \eqref{cvsufficiency} and \eqref{strsufficiency} ensure the existence of neighborhoods $U$ of $\ox$ and $V$ of $0$ on which we have \eqref{PSDvarphi1} and \eqref{PSDvarphi2}, respectively. Likewise, the pointbased condition \eqref{varsufficiency} yields \eqref{PSDvarphi3}. Applying finally Theorems~\ref{2ndconvexsub}, \ref{2ndstrongconvex} and \ref{pointPD} verifies, respectively, the fulfillment of assertions (i), (ii), and (iii) of this theorem.
\end{proof}

Note that the second-order conditions obtained in Theorem~\ref{vrcoderivativenotfullrank} imply the variational sufficiency and strong variational sufficiency for local optimality in \eqref{composite} while not characterize them. The next theorem reveals additional assumptions on the problem data that ensure {\em complete characterizations} of these properties in composite optimization. We are dealing below with {two crucial subclasses} of {\em fully amenable} compositions $\psi\circ g$ in \eqref{composite}. The first subclass is generated by extended-real-valued {\em piecewise linear} functions $\psi$ in the sense of \cite[Definition~2.47]{Rockafellar98}. The second subclass is formed by the functions $\psi\colon\R^m\to\oR$ defined as
\begin{equation}\label{enlp}
\psi(z):=\sup_{v\in P}\bigg\{\langle v,z\rangle-\frac{1}{2}\big\langle Qv,v\big\rangle\bigg\},
\end{equation}
where $P\subset\R^m$ is a nonempty polyhedral set, and where $Q\in\R^{m\times m}$ is a symmetric positive-semidefinite matrix. Functions of this type appear in problems of {\em extended linear-quadratic programming}; see \cite{Rockafellar98}. {To deal with class \eqref{enlp}}, recall that a mapping $f\colon\R^n\to\R^m$ is {\em open} around $\ox$ if for any neighborhood $U$ of $\ox$ there exists a neighborhood $V$ of $f(\ox)$ such that $V\subset f(U)$. Here are the aforementioned characterizations of {variational and strong variational sufficiency} for local optimality in \eqref{composite}, which have the same form as in Theorem~\ref{vrcoderivative} (including the uniqueness of multipliers) {without imposing the full rank assumption.}

\begin{Theorem}[\bf characterizations of variational and strong variational sufficiency in composite fully amenable optimization]\label{vs-amen} In addition to the assumptions of Theorem~{\rm\ref{vrcoderivativenotfullrank}}, suppose that 

{\bf(a)} either $\psi$ is piecewise linear,

{\bf(b)} or $\psi$ is of class \eqref{enlp}, $Q$ is positive-definite, and the inner mapping $g$ is open around $\ox$.\\[1ex] 
Then all the three characterizations {\rm(i)--(iii)} of Theorem~{\rm\ref{vrcoderivative}} hold.
\end{Theorem}
\begin{proof} We proceed as in the proof of Theorem~\ref{vrcoderivativenotfullrank} {by observing that} the additional assumptions imposed either in (a), or in (b) of this theorem ensure that the inclusion in \eqref{chainrulecomamenable} holds as equality with $\Lambda(x,w)$ being a singleton. This is due to the second-order subdifferential chain rules of the equality type derived under (a) and (b) in \cite[Theorem~4.3]{mr} and  \cite[Theorem~4.5]{mr}, respectively.
\end{proof}

We now compare our {\em second-order subdifferential approach} to variational and strong variational sufficiency for local optimality in composite optimization \eqref{composite} with Rockafellar's developments in the very recent paper \cite{roc} to characterize {\em strong} variational sufficiency in \eqref{composite} {via {\em quadratic bundles} defined in \eqref{qb}.}

\begin{Remark}[\bf strong variational sufficiency via augmented Lagrangians and quadratic bundles]\label{rock-vs} {\rm Rockafellar's approach \cite{roc} to variational and strong variational sufficiency for local optimality in \eqref{composite} with the l.s.c.\ convex function $\psi$    is based on exploring a certain ``hidden convexity" (of a local convex-concave type) of the corresponding {\em augmented Lagrangian function}. 
		In this way, Rockafellar established a criterion for {\em strong} variational sufficiency via \textit{Hessian bundles} of augmented Lagrangians; see \cite[Theorem~3]{roc}. Moreover, using the other type of generalized second-order derivatives called \textit{quadratic bundles}, it is shown in \cite[Theorem~5]{roc} that the condition
		\begin{equation}\label{2nd-der}
			q \in \text{\rm quad}\;\psi(g(\ox),\oy),\; w \ne 0 \Longrightarrow \frac{1}{2}\big\la\nabla^2_{xx}L(\ox,\oy),w,w\big\ra+q\big(\nabla g(\ox)w\big)>0.
		\end{equation}
		via the standard Lagrangian of \eqref{composite}, characterizes the strong variational sufficiency for local optimality in \eqref{composite}. In the case where $\psi$ is nonconvex, it is possible to derive the characterization of strong variational sufficiency via the connection to SCD mappings (cf,\ \cite[Proposition 3.33]{go22}), which is also discussed in Remark \ref{discussSCD}.  To the best of our knowledge, the characterization of the merely \textit{variational sufficiency} for local optimality in \eqref{composite} via quadratic bundles has not been done. Our second-order subdifferential approach is independent of the aforementioned approach of Rockafellar. Let us emphasize that our second-order subdifferential results in Theorems~\ref{vrcoderivative}, \ref{vrcoderivativenotfullrank}, and \ref{vs-amen} {\em completely characterize both} variational and strong variational
		sufficiency in \eqref{composite}  with and without the convexity of the function $\psi$ therein.}
\end{Remark}\vspace*{-0.22in}

\section{Applications to Nonlinear Programming}\label{vs-nlp}\vspace*{-0.05in}
\setcounter{equation}{0}

As an illustration of the general results of Section~\ref{sec:varsuf}, we provide here their direct applications to characterizing variational and strong variational sufficiency for local optimality in classical problems on nonlinear programming with ${\cal C}^2$-smooth data. The characterizations obtained below are expressed entirely in terms the problem data meaning that the second-order subdifferentials in the results of Section~\ref{sec:varsuf} are {\em explicitly calculated}. 

The conventional model of nonlinear programming (NLP) is formulated as follows:
\begin{equation}\label{NLPproblem}
\min\quad\ph_0(x)\;\text{ subject to }\;
\begin{cases}
\ph_i(x)\le 0&\text{for }\;i=1,\ldots,s,\\
\ph_i(x)=0&\text{for }\;i=s+1,\ldots,m,
\end{cases}
\end{equation}
where $\ph_i$, $i=0,\ldots,m$, are $\mathcal{C}^2$-smooth functions around the references points. Problem \eqref{NLPproblem} can be obviously written in the form of composite optimization \eqref{composite} with $\psi=\delta_\O$, where $\O$ is given by 
\begin{equation}\label{NLP}
\O:=\big\{u\in\R^m\;\big|\;u_i\le 0\;\text{ for }\;i=1,\ldots,s\;\mbox{ and }\;u_i =0\;\text{ for }\;i=s+1,\ldots,m\big\},
\end{equation}
and where $g(x):=(\ph_1(x),\ldots,\ph_m(x))$. Define the {\em Lagrangian function} $L:\R^n\times\R^m\to\R$ by 
\begin{equation}\label{Lagrangefunction}
L(x,y):= \ph_0(x)+\langle y,g(x)\rangle=\ph_0(x)+y_1\ph_1(x)+\ldots+y_m\ph_m(x)  \quad\text{for all }\;x\in\R^n,\;y\in\R^m
\end{equation}
and for each $(x,y)\in\R^n\times\R^m$ consider the subspace
\begin{equation}\label{subspaceS}
S(x,y):=\big\{w\in\R^n\;\big|\;\langle\nabla\ph_i(x),w\rangle=0\;\text{ for }\;i\in I_+(x,y)\cup\{s+1,\ldots,m\}\big\}
\end{equation}
together with the index collections 
\begin{equation}\label{index}
I_+(x,y):=\big\{i\in I(x)\;\big|\;y_i>0\big\}\;\text{ and }\;I(x):=\big\{i\in\{1,\ldots,s\}\;\big|\;\ph_i(x)=0\big\}.
\end{equation}

First-order and second-order subdifferential calculations for the indicator function of set \eqref{NLP} can be deduced from various sources and represented in different forms. Here we present a simple direct derivation and present the corresponding formulas in the explicit and convenient way used in what follows.

\begin{Lemma}[\bf explicit subdifferential calculations for NLP]\label{calcNLP2nd} Let the set $\O$ be taken from \eqref{NLP}. Then 
\begin{equation}\label{normalconeKs}
\partial\delta_\O(z)=\partial^\infty\delta_\O(z)=N_\O(z)= F_1(z_1)\times\ldots\times  F_m(z_m)\;{\text{ for all }\;z\in \Omega,}
\end{equation} 
where each multifunction $F_i\colon\R\tto\R$ is given by
\begin{equation*}
F_i(t):=\begin{cases}
[0,\infty)&\text{if}\quad t=0,\;1\le i\le s,\\
\{0\} &\text{if}\quad t<0,\;1\le i\leq s,\\
\R &\text{if}\quad t =0,\; s+1\le i\le m,\\
\emp &\text{otherwise}.
\end{cases} 
\end{equation*}
Furthermore, for every $(z,y)\in\gph\partial\delta_\O$ we have 
\begin{equation}\label{secorderOmega}
\partial^2\delta_\O(z,y)(v)=\big\{u\in\R^m\;\big|\;(u_i,-v_i)\in G_i(z_i,y_i),\;i=1,\ldots,m\big\}\;\mbox{ for all }\;v\in\R^m,
\end{equation} 
where the multifunctions $G_i\colon\R^2\tto\R^2$, $i=1.\ldots,m$, are given by
\begin{equation*}
 G_i(t,p):=\begin{cases}
(\mathbb{R}_+\times\mathbb{R}_{-})\cup\big(\mathbb{R}\times\{0\}\big)\cup\big(\{0\}\times\mathbb{R}\big)&\text{if}\quad t =0,\;p=0,\;1\le i\le s,\\
\mathbb{R}\times\{0\}&\text{if}\quad t=0,\;p>0,\;1\le i\le s,\\
\{0\}\times\mathbb{R}&\text{if}\quad t<0,\;p=0,\;1\le i\le s, \\
\mathbb{R}\times\{0\}&\text{if}\quad t=0,\;s+1\le i\le m,\\
\emptyset &\text{otherwise}.
\end{cases}  
\end{equation*} 
\end{Lemma}
\begin{proof} Recall first by \cite[Proposition~1.79]{Mordukhovich06} that
$$
\partial\delta_\O(z)=\partial^\infty\delta_\O(z)=N_\O(z)\quad\text{for all }\;z\in\O.
$$
It follows from $\O=\Omega_1\times\ldots\times\Omega_m$ with
$$
\Omega_i:=\begin{cases}
(-\infty,0]&\text{if }\quad i=1,\ldots,s,\\
\{0\}&\text{if }\quad i=s+1,\ldots,m
\end{cases}
$$
that we have the representation
$$
\delta_\O(z)=\delta_{\Omega_1}(z_1)+\ldots+\delta_{\Omega_m}(z_m)\quad\text{for all }\;z=(z_1,\ldots,z_m)\in\R^m.
$$
This easily leads us to the subdifferential expression
\begin{equation}\label{sepindicator}
\partial\delta_\O(z)=\partial\delta_{\Omega_1}(z_1)\times\ldots\times \partial\delta_{\Omega_m}(z_m)\quad\text{for all }\;z\in\O.  
\end{equation} 
Since $\Omega_i$ is convex for any $i\in\{1,\ldots,m\}$, it is not difficult to check that 
\begin{equation}\label{norOmegai}
\partial\delta_{\Omega_i}(t)=N_{\Omega_i}(t)=F_i(t)\quad\text{whenever }\;t\in\R,\;i =1,\ldots,m. 
\end{equation} 
Combining \eqref{sepindicator} and \eqref{norOmegai} justifies the first-order formulas in \eqref{normalconeKs}. 

To verify finally the second-order subdifferential formula \eqref{secorderOmega}, observe that $N_{{\rm\small gph}\,\partial\delta_{\Omega_i}}=G_i\quad\text{for all }\;i=1,\ldots,m$. This allows us to deduce from \cite[Theorem~4.3]{BorisOutrata} the representation
\begin{equation*}
\partial^2\delta_\O(z,y)(v)=\big\{u\in\R^m\;\big|\;\big(u_i,-v_i\big)\in N_{{\rm\small gph}\,\partial\delta_{\Omega_i}}(z_i,y_i),\;i=1,\ldots,m\big\},
\end{equation*}
which therefore justifies the fulfillment of \eqref{secorderOmega} and completes the proof of the lemma.
\end{proof}

Thanks to Lemma~\ref{calcNLP2nd}, we have the following explicit formula for the set of multipliers \eqref{Lagrangemultiplier} in problem \eqref{NLPproblem}:
\begin{equation}\label{Lmulti2}
\Lambda(x,v)=\big\{y\in\R_+^s\times\R^{m-s}\;\big|\;v=\nabla_x L(x,y),\;\langle y, g(x)\rangle=0\big\}
\end{equation} 
with $g=(\ph_1,\ldots,\ph_m)$ and the Lagrangian function $L$ defined in \eqref{Lagrangefunction}.\vspace*{0.05in}

Next we recall the two well-known constraint qualification conditions in NLP. Given a feasible solution $\ox$ to \eqref{NLPproblem}, it is said that the {\em linear independent constraint qualification} (LICQ) holds at $\ox$ of \eqref{NLPproblem} if the vectors
\begin{equation}\label{linearICQ}
\nabla\ph_i(\ox)\;\mbox{ for }\;i\in I(\ox)\;\text{and }\;\nabla\ph_i(\ox)\;\mbox{ for }\;i=s+1,\ldots,m\;\text{ are linearly independent},
\end{equation} 
where $I(\ox)$ stands for the collection of active inequality constraint indices at $\ox$ taken from \eqref{index}. We say that the {\em positive linear independence constraint qualification} (PLICQ) is satisfied at $\ox$ if 
\begin{equation}\label{linearPICQ}
 {\bigg[\disp\sum_{i\in I(\ox)\cup\{s+1,\ldots,m\}}\al_i\nabla\ph_i(\ox)=0,\;\al_i\ge 0,\;i \in I(\ox)  \bigg] \Longrightarrow\big[\al_i=0\;\mbox{ for all }\; i \in I(\ox)\cup \{s+1,\ldots,m\} \big].}
\end{equation} 
Note that both of these constraint qualifications are robust with respect to small perturbations of $\ox$. In fact, \eqref{linearPICQ} is a dual form of the {\em Mangasarian-Fromovitz constraint qualification} (MFCQ) at $\ox$.\vspace*{0.05in}

By using Lemma~\ref{calcNLP2nd}, we show now that the first-order qualification condition \eqref{1stqualify} and the second-order qualification condition \eqref{SOQC} are equivalent to PLICQ and LICQ at $\ox$, respectively,

\begin{Lemma}[\bf equivalent descriptions of first- and second-order qualification conditions]\label{lemma:cq} Let $\ox$ be a feasible solution for NLP \eqref{NLPproblem} written in the framework of composite optimization \eqref{composite} with
\begin{equation}\label{nlp}
\psi:=\dd_\O\;\mbox{ and }\;g:=(\ph_1,\ldots,\ph_m),
\end{equation}
where $\O$ is taken from \eqref{NLP}. Then the following assertions hold:

{\bf(i)} We have the representation
\begin{equation}\label{explicit1st}
\partial^\infty\psi\big(g(\ox)\big)\cap\ker\nabla g(\ox)^*=\left\{u\in\R^m\;\bigg|\; \sum_{i\in J}u_i\nabla\ph_i(\ox)=0,\;u_i\ge 0\;\forall  {i \in I(\ox)}\;\mbox{ and }\;u_i= 0 \;\forall i\notin J\right\},
\end{equation}
where $J:=I(\ox)\cup\{s+1,\ldots,m\}$. Consequently, the first-order qualification condition \eqref{1stqualify} at $\ox$ is equivalent to the fulfillment of PLICQ \eqref{linearPICQ} at this point.

{\bf(ii)} Whenever $\oy\in\partial\psi(g(\ox))$, we have the representation
\begin{equation}\label{explicitSOQC}
\partial^2\psi\big(g(\ox),\oy\big)(0)\cap\ker\nabla g(\ox)^*=\left\{u\in\R^m\;\bigg|\;\sum_{i\in J}u_i\nabla\ph_i(\ox)=0,\;u_i=0\;\text{ for all }\;i\notin J\right\}
\end{equation}
with $J$ taken from {\rm(i)}. Consequently, the second-order qualification condition \eqref{SOQC} is equivalent to the fulfillment of LICQ \eqref{linearICQ} at this point.
\end{Lemma}

\begin{proof} To verify assertion (i), deduce from Lemma~\ref{calcNLP2nd} that
\begin{equation}\label{uin1stsub}
u\in\partial^\infty\psi\big(g(\ox)\big)\iff u_i\ge 0\quad\text{for all }\; {i\in I(\ox)}\quad\text{and }\;u_i=0\quad\text{for all }\;i\notin J.
\end{equation} 
Furthermore, it follows from the definitions that
\begin{equation}\label{uinker1}
u\in\ker\nabla g(\ox)^*\iff\nabla g(\ox)^*u=0\iff  {\sum_{i=1}^m u_i\nabla\varphi_i(\ox)=0.}
\end{equation}
Combining \eqref{uin1stsub} and \eqref{uinker1} justifies \eqref{explicit1st}, which yields the claimed equivalence between \eqref{1stqualify} and \eqref{linearPICQ}.

To proceed with the verification of (ii), pick $\oy\in\partial\psi(g(\ox))$ and get from Lemma~\ref{calcNLP2nd} that
\begin{equation*}
u\in\partial^2\psi\big(g(\ox),\oy\big)(0)\iff u_i=0\quad\text{for all }\;i\notin J.
\end{equation*} 
Using this together with \eqref{uinker1} gives us \eqref{explicitSOQC}, which immediately implies the claimed equivalence in (ii). 
\end{proof}

Reformulating the NLP problem \eqref{NLPproblem} in the form of composite optimization \eqref{composite} with $\psi$ and $g$ taken from \eqref{nlp}, observe that the stationary condition $0\in\partial\ph(\ox)$ can be equivalently written as the {\em KKT system}
\begin{equation}\label{1stKKT}
\nabla_x L(\ox,\oy)=0\quad\text{and }\;\langle\oy,g(\ox)\rangle=0\quad\text{for some }\;\bar{y}\in\R_+^s\times\R^{m-s}
\end{equation}
provided that the first-order qualification condition \eqref{1stqualify}, i.e., PLICQ in our case, holds. Indeed, it follows from the subdifferential sum and chain rules taken from \cite[Proposition~1.107(ii) and  Theorem~3.41(ii)]{Mordukhovich06}, respectively, with the usage of Lemma~\ref{calcNLP2nd} ensuring representation \eqref{Lmulti2}.\vspace*{0.05in}

The next theorem establishes second-order characterizations of variational and strong variational sufficiency for local optimality in nonlinear programming entirely in terms of the given data of \eqref{NLPproblem}.

\begin{Theorem}[\bf variational and strong variational sufficiency for local optimality in NLP]\label{vrNLPsufficiency} Let $\ox$ be a feasible solution
to the NLP problem \eqref{NLPproblem} satisfying the first-order optimality condition \eqref{1stKKT} under the fulfillment of LICQ at $\ox$, and let $S(x,y)$ be defined in \eqref{subspaceS}. Then we have the following assertions:

{\bf(i)} The variational sufficiency for local optimality in \eqref{NLPproblem} holds at $\ox$ if and only if there exist neighborhoods $U$ of $\ox$, $V$ of $0$ such that 
\begin{equation}\label{PSDNLP}
\langle\nabla^2_{xx}L(x,y)w,w\rangle\ge 0\quad\text{whenever }\;x \in U,\;v\in V,\;\mbox{ and }\;w\in S(x,y),
\end{equation}
where $y\in\R^s_+\times\R^{m-s}$ is a unique solution to the system $\nabla_x L(x,y)=v$ and $\la y,g(x)\ra=0$ with $g$ from \eqref{nlp}.

{\bf(ii)} The strong variational sufficiency for local optimality in \eqref{NLPproblem} holds at $\ox$ with modulus $\sigma>0$ if and only if there exist neighborhoods $U$ of $\ox$ and $V$ of $0$ such that 
\begin{equation}\label{PDNLP}
\langle\nabla^2_{xx}L(x,y)w,w\rangle\ge\sigma\|w\|^2\quad\text{whenever }\;x\in U,\;v\in V,\;\mbox{ and }\;w\in S(x,y)
\end{equation}
with the unique vector $y\in\R^s_+\times\R^{m-s}$ defined as in {\rm(i)}.

{\bf(iii)} The strong variational sufficiency for local optimality in \eqref{NLPproblem} holds at $\ox$ if and only if 
\begin{equation}\label{pointPDNLP}
\langle\nabla^2_{xx}L(\ox,\oy)w,w\rangle>0\quad\text{whenever }\;\oy\in\Lambda(\ox,0)\;\mbox{ and }\;w\in S(\ox,\oy)\setminus\{0\},
\end{equation} 
where $\oy$ is a unique solution to the KKT system \eqref{1stKKT}.
\end{Theorem}

\begin{proof}  By Lemma~\ref{lemma:cq}(ii), the fulfillment of LICQ at $\ox$ is equivalent to the second-order qualification condition \eqref{robustSOQC}, and thus \eqref{1stqualify} is satisfied by Proposition~\ref{2ndqualifyimplies1st} with $\psi=\dd_\O$. This also follows, in the NLP case, from the PLICQ description of \eqref{1stqualify} in Lemma~\ref{lemma:cq}(i). As discussed above, condition \eqref{1stKKT} is equivalent to the stationary $0\in\partial\ph(\ox)$ of $\ox$ for 
$\varphi=\ph_0+\dd_\O\circ g$ with $\O$ from \eqref{NLP} and $g=(\ph_1,\ldots,\ph_m)$. 

It is easy to see that the imposed LICQ corresponds to the full rank assumption for $g$ from Theorem~\ref{vrcoderivative}, and thus we can apply the variational and strong variational sufficiency characterizations of that theorem for the composite problem \eqref{composite} to the case of NLP \eqref{NLPproblem}. Note that the convex l.s.c.\ function $\psi=\dd_\O$ for $\O$ from \eqref{NLP} is automatically continuously prox-regular. Observe also that this function is even piecewise linear, and thus we can equally apply Theorem~\ref{vs-amen}(a) to the NLP setting \eqref{NLPproblem}. It remains to express the general variational and strong variational sufficiency characterization in composite optimization explicitly in terms of the NLP data.\vspace*{0.03in}

To proceed, we employ Lemma~\ref{calcNLP2nd}, which tells us that the inclusion $ {u \in\partial^2\delta_\O(g(x),y)(\nabla g(x)w)}$ for each $u\in\R^m$, $w\in\R^n$, $x\in \R^n$, and $y\in\partial\delta_\O(g(x))$ amounts to saying that
\begin{equation}\label{explicit2ndeltaK}
\begin{cases}
u_i=0&\text{if }\;i\in\big\{1,\ldots,s\big\}\setminus I(x),\\
\langle\nabla\ph_i(x),w\rangle=0&\text{if }\;i\in\big\{s+1,\ldots,m\big\}\cup I_+(x,y),\\
\langle\nabla\ph_i(x),w\rangle\ge 0,\;u_i\ge 0& \text{if }\;i\in I(x)\setminus I_+(x,y). 
\end{cases}
\end{equation} 

First we verify that the conditions in \eqref{PSDNLP}, \eqref{PDNLP}, and \eqref{pointPDNLP} yield the variational and strong variational sufficiency with and without modulus $\sigma>0$ for local optimality in \eqref{NLPproblem}, respectively. To this end, suppose that \eqref{PSDNLP} holds and pick any $x\in U,\;v\in V,\;w\in\R^n$, and $u\in\partial^2\psi(g(x),y)(\nabla g(x)w)$ in the setting of \eqref{nlp} with the unique solution $y$ of the corresponding KKT system. We aim to show that 
\begin{equation}\label{PSDsufficiencyNLP}
\langle\nabla^2\ph_0(x)w,w\rangle+\langle\nabla^2\langle y,g\rangle(x)w,w\rangle+\langle u,\nabla g(x)w\rangle\ge 0. 
\end{equation}
Indeed, it follows from \eqref{explicit2ndeltaK} that
\begin{equation}\label{uFxw}
u_i\langle\nabla\ph_i(x),w\rangle\ge 0\quad\text{ for all }\;i=1,\ldots,m\;\mbox{ and }\;w\in S(x,y),
\end{equation}
and thus we deduce from \eqref{PSDNLP} that $\langle\nabla_{xx}^2 L(x,y)w,w\rangle\ge 0$. Combining the latter with \eqref{uFxw} gives us \eqref{PSDsufficiencyNLP}, and therefore the claimed variational sufficiency for local optimality in \eqref{NLPproblem} holds by Theorem~\ref{vrcoderivativenotfullrank}. Arguing similarly to the arguments above, we verify that the strong variational sufficiency for local optimality in \eqref{NLPproblem} with and without modulus $\sigma>0$ is also satisfied under the conditions in \eqref{PDNLP} and \eqref{pointPDNLP}, respectively. \vspace*{0.03in}

Next we prove the converse, i.e., that the variational and strong variational sufficient conditions with and without modulus $\sigma>0$ for local optimality in \eqref{NLPproblem} ensure the fulfillment of the conditions in \eqref{PSDNLP}, \eqref{PDNLP}, and \eqref{pointPDNLP}, respectively. To proceed with the case of variational sufficiency in (i), we get by this property for \eqref{NLPproblem} at $\ox$ that there exist neighborhoods $U$ of $\ox$ and $V$ of $0$ such that \eqref{cvsufficiency} holds {due to Theorem~\ref{vs-amen}(a)}. Taking any $x\in U,\;v\in V$, and $w\in S(x,y)$ with the unique solution $y$ of the corresponding KKT system, let us verify the fulfillment of \eqref{PSDNLP}. Indeed, choose $u\in\R^m$ such that $u_i=0$ for any $i\in\{1,\ldots,s\}\setminus I_+(x,y)$. It follows from \eqref{explicit2ndeltaK} that $u\in\partial^2\psi(g(x),y)(\nabla g(x)w)$ for $\psi$ and $g$ in \eqref{nlp}. Moreover, the choice of $u$ implies that
\begin{equation}\label{firstsum}
\sum_{i\in\{1,\ldots,s\}\setminus I_+(x,y)}u_i\langle\nabla\ph_i(x),w\rangle=0.
\end{equation} 
Since $w\in S(x,y)$, we readily get that
\begin{equation}\label{secondsum}
\sum_{i\in\{s+1,\ldots,m\}\cup I_+(x,y)}u_i\langle\nabla\ph_i(x),w\rangle=0. 
\end{equation}
Combining \eqref{firstsum} and \eqref{secondsum} tells us that
$$
\langle u,\nabla g(x)w\rangle=\sum_{i=1}^m u_i\langle\nabla\ph_i(x),w\rangle=0.
$$
Using the latter together with \eqref{cvsufficiency} leads us to 
$$
\langle\nabla^2_{xx} L(x,y)w,w\rangle=\langle\nabla^2\ph_0(x)w,w\rangle+\langle \nabla^2\langle y,g\rangle(x)w,w\rangle\ge 0,
$$
which justifies \eqref{PSDNLP}. The verifications for (ii) and (iii) are similar, and thus we are done with the proof.
\end{proof}

As an illustration of the obtained characterizations of variational and strong variational sufficiency, we now recover some known optimality conditions for nonlinear programs obtained before by using different approaches.  In this way, we also establish a {\em new condition} discussed in Remark~\ref{rem:opt}(i) below.

\begin{Corollary}[\bf second-order optimality conditions from variational sufficiency in NLP]\label{opt-nlp} Let $\ox$ be a feasible solution to NLP \eqref{NLPproblem} satisfying the stationary condition \eqref{1stKKT} under the fulfillment of LICQ at $\ox$. Then we have the following assertions:

{\bf(i)} $\ox$ is a local minimizer of \eqref{NLPproblem} if there exist neighborhoods $U$ of $\ox$ and $V$ of $0$ such that 
\begin{equation}\label{PSDlocalminimizer}
\langle\nabla^2_{xx} L(x,y)w,w\rangle\ge 0\quad\text{for all }\;x\in U,\;v\in V,\;\mbox{ and }\;w\in S(x,y),
\end{equation}
where $y\in\R^s_+\times\R^{m-s}$ is as in {\rm(i)} of Theorem~{\rm\ref{vrNLPsufficiency}}.
    
{\bf(ii)} $\ox$ is a tilt-stable local minimizer of \eqref{NLPproblem} with modulus $\kappa>0$ if and only if there exist neighborhoods $U$ of $\ox$ and $V$ of $0$ such that 
\begin{equation}\label{PDNLPtilt}
\langle\nabla^2_{xx}L(x,y)w,w\rangle\ge\frac{1}{\kappa}\|w\|^2\quad\text{for all }\; x\in U,\;v \in V,\;\mbox{ and }\;w\in S(x,y),
\end{equation}
where $y\in\R^s_+\times\R^{m-s}$ is taken from {\rm(i)}.

{\bf(iii)} $\ox$ is a tilt-stable local minimizer of \eqref{NLPproblem} if and only if 
\begin{equation}\label{pointPDtilt}
\langle\nabla^2_{xx}L(\ox,\oy)w,w\rangle>0\quad\text{for all }\;w\in S(\ox,\oy),  \end{equation}
where $\oy$ is a unique solution to the KKT system \eqref{1stKKT}.
\end{Corollary}
\begin{proof} Theorem \ref{vrNLPsufficiency}  tells us that condition \eqref{PSDlocalminimizer} yields the variational sufficiency for local optimality of $\ox$ in the composite problem \eqref{composite}
with $\psi$ and $g$ taken from \eqref{nlp}. This means that $\varphi$ in \eqref{composite} is variationally convex at $\ox$ for $0$. As discussed in Remark~\ref{sttilt}(i), $\ox$ is a local minimizer of $\varphi$, which verifies assertion (i). Note that $\varphi$ is continuously prox-regular on $\R^n$ due to \cite[Exercise~10.26 and Proposition~13.32]{Rockafellar98}, and thus $\ox$ is a tilt-stable local minimizer of \eqref{NLPproblem} with modulus $\kappa>0$ if and only if $\varphi$ is variationally convex at $\ox$ for $0$ with modulus $\kappa^{-1}$ due to Lemma~\ref{equitiltstr}. Therefore, assertions (ii) and (iii) are implied immediately by Theorem~\ref{vrNLPsufficiency}. 
\end{proof}

The following discussions shed some light on the obtained results and related developments.

\begin{Remark}[\bf discussions on optimality and tilt-stability conditions]\label{rem:opt}{\rm\,

{\bf(i)} Observe that \eqref{PSDlocalminimizer} is a ``nonstrict" neighborhood {\em second-order sufficient condition} for usual (nonstrict) local minimizers that seems to be new. We derive it as a direct consequence of the established {\em characterization of variational sufficiency}. Note that the characterizations of the variational sufficiency achieved in Section~\ref{sec:varsuf} and the explicit computations of second-order subdifferentials mentioned in Section~\ref{sec:prelim} allow us to obtain similar optimality conditions for much more general classes of composite optimization problems.

{\bf(ii)} A {\em neighborhood} characterization of tilt stability in \eqref{PDNLPtilt} was first obtained under LICQ in \cite[Theorem~5.3]{MorduNghia13} for NLP in Hilbert spaces by using another technique. Then this characterization was further developed in \cite{chn,gm,MorduNghia} for NLP in finite dimensions by using significantly less restrictive constraint qualifications.

{\bf(iii)} The {\em pointbased} characterization \eqref{pointPDtilt} of tilt-stable minimizers for finite-dimensional nonlinear programs under LICQ was first derived in \cite[Theorem~5.2]{mr}. Advanced pointbased conditions for tilt stability in NLP without imposing LICQ were developed in \cite{chn,gm} in explicit forms different from \eqref{pointPDtilt}.

{\bf(iv)} Explicit neighborhood and pointbased characterizations of tilt stability were established not only for {\em polyhedral} problems of composite optimization  (as NLP and the like), but also for {\em nonpolyhedral} optimization problems of {\em conic programming} including {\em second-order cone programs} and {\em semidefinite} ones; see \cite{bgm,mnr,mor,mos} and the references therein. These results are instrumental for our understanding of {\em strong variational sufficiency} in constrained optimization, while the study of merely {\em variational sufficiency} is a major {\em open question}.}
\end{Remark}\vspace*{-0.22in}

\section{Conclusions and Future Research}\label{sec:conclusion}\vspace*{-0.05in}
\setcounter{equation}{0}

This paper contributes to the study and applications of the powerful Rockafellar's notions of variational convexity of extended-real-valued functions and variational sufficiency for local optimality as well as of their strong counterparts. The main novelty of our results is in equivalent reductions of these notions to the conventional local convexity and strong local convexity of Moreau envelopes with their subsequent characterizations via second-order subdifferentials of variational analysis. The obtained characterizations allow us to efficiently study variational and strong variational sufficiency in problems of composite optimization with further applications to nonlinear programming.\vspace*{0.03in}

Some topics of our future research have been already mentioned in the text. They include, in particular, {\em numerical methods} that benefit from the local convexity/local strong convexity of Moreau envelopes of variationally convex/ variationally strongly convex functions (see Remark~\ref{vc-newton}); explicit characterizations of {\em variational sufficiency} vs.\ strong variational sufficiency of polyhedral and nonpolyhedral optimization problems (see Remark~\ref{rem:opt}), etc. In the other lines of our future research, we plan to establish {\em graphical derivative} characterizations of variational convexity and strong variational convexity for extended-real-valued functions with deriving the corresponding characterizations of variational and strong variational sufficiency in problems of composite optimization with subsequent applications to particular classes of constrained optimization problems.\vspace*{0.1in}
		
{\bf Acknowledgements}. {The authors are very grateful to two anonymous referees and Nghia Tran for their constructive remarks and suggestions, which allowed us to significantly improve the original presentation.} The second author greatly appreciates helpful discussions with Terry Rockafellar.


\small
\begin{thebibliography}{99}
\bibitem{Bauschke} {\sc H. H. Bauschke and P. L. Combettes}, {\em Convex Analysis and Monotone Operator Theory in Hilbert Spaces}, 2nd edition, Springer, New York, 2017.

\bibitem{bgm} {\sc M. Benko, H. Gfrerer and B. S. Mordukhovich}, {\em Characterizations of tilt-stable minimizers in second-order cone programming}, SIAM J. Optim., 29 (2019), pp.\ 3100--3130.

\bibitem{cwb}  {\sc E. J. Cand\'es, M. B. Wakin and S. P. Boyd},  {\em  Enhancing sparsity by reweighted $\ell^1$ minimization}, J. Fourier Anal. Appl., 14 (2008), pp.\ 877--905.


\bibitem{ChieuChuongYaoYen} {\sc N. H. Chieu, T. D. Chuong, J.-C. Yao and N. D. Yen}, {\em Characterizing convexity of a function by its Fr\'echet and limiting second-order subdifferentials}, Set-Valued Var. Anal., 19 (2011), pp.\ 75--96.

\bibitem{chn} {\sc N. H. Chieu, L. V. Hien and T. T. A. Nghia}, {\em Characterizations of tilt stability via subgradient graphical derivative with applications to nonlinear programming}, SIAM J. Optim., 28 (2018), pp.\ 2246--2273.

\bibitem{ChieuHuy11} {\sc N. H. Chieu and N. Q. Huy}, 
{\em Second-order subdifferentials and convexity of real-valued functions}, Nonlinear Anal., 74 (2011), pp.\ 154--160. 

\bibitem{clmn} {\sc N. M. Chieu, G. M. Lee, B. S. Mordukhovich and T. T. A. Nghia}, {\em  Coderivative characterizations of maximal monotonicity for set-valued mappings}, J. Convex Anal., 23 (2016), pp.\ 461--480.

\bibitem{ChieuLee17} {\sc N. H. Chieu, G. M. Lee and N. D. Yen}, {\em  Second-order subdifferentials and optimality conditions for ${\cal C}^1$-smooth optimization problems},  Appl. Anal. Optim., 1 (2017), pp.\ 461--476.

\bibitem{chhm} {\sc G. Colombo, R.  Henrion, N. D.  Hoang and B. S. Mordukhovich}, {\em Optimal control of the sweeping process over polyhedral controlled sets}, J. Diff. Eqs., 260 (2016), pp.\ 3397--3447.

\bibitem{dsy} {\sc C. Ding, D. Sun and J. J. Ye}, {\em First-order optimality conditions for mathematical programs  with semidefinite and cone complementarity constraints}, Math. Program., 147 (2014), pp.\ 539--579.

 
\bibitem{dslmh} {\sc W. Dong, G. Shi, X. Li, Y. Ma and F. Huang},  {\em Compressive sensing via nonlocal low-rank regularization}, IEEE Trans. Image Process., 23 (2014), pp.\ 3618--3632.

 
\bibitem{dr} {\sc A. L. Dontchev and R. T. Rockafellar}, {\em Characterizations of strong regularity for variational inequalities over polyhedral convex sets}, SIAM J. Optim., 6 (1996), pp.\ 1087--1105.
	
\bibitem{dl} {\sc D. Drusvyatskiy and A. S. Lewis, {\em Tilt stability, uniform quadratic growth, and strong metric regularity of the subdifferential}, SIAM J. Optim., 23 (2013)}, pp.\ 256--267.

\bibitem{dmn} {\sc D. Drusvyatskiy, B. S. Mordukhovich and T. T. A. Nghia}, {\em Second-order growth, tilt stability, and metric regularity of the subdifferential}, J. Convex Anal., 21 (2014), pp.\ 1165--1192.
			
\bibitem{gm} {\sc H. Gfrerer and B. S. Mordukhovich}, {\em Complete
characterization of tilt stability in nonlinear programming under
weakest qualification conditions}, SIAM J. Optim., 25 (2015), pp.\ 2081--2119.


\bibitem{go22} {\sc H. Gfrerer and J. V. Outrata}, {\em On $($local$)$ analysis of multifunctions via subspaces contained in graphs of
generalized derivatives}, J. Math. Anal. Appl., 508 (2022), 125895. 

			
\bibitem{hmn} {\sc R. Henrion, B. S. Mordukhovich and N. M. Nam}, {\em Second-order analysis of polyhedral systems in finite and infinite dimensions with applications to robust stability of variational inequalities}, SIAM J. Optim., 20 (2010), pp.\ 2199--2227.

\bibitem{hos} {\sc R. Henrion, J. V. Outrata and T. Surowiec}, {\em Analysis of $M$-stationary points to an EPEC modeling ologopolistic competition in an electricity spot market}, ESAIM: Control Optim. Calc. Var., 18 (2012), pp.\ 295--317.

\bibitem{hr} {\sc R. Henrion and W. R\"omisch}, {\em On $M$-stationary points for a stochastic equilibrium problem under equilibrium constraints in electricity spot market modeling}, Appl. Math., 52 (2007), pp.\ 473--494.

\bibitem{hl04} {\sc J.-B. Hiriart-Urruty and  C. Lemar\'echal}, {\em Fundamentals of Convex Analysis}, Springer, Berlin, 2004.

\bibitem{BorisKhanhPhat} {\sc P. D. Khanh, B. S. Mordukhovich and V. T. Phat}, {\em A generalized Newton method for subgradient systems}, Math. Oper. Res., to appear (2022), arXiv:2009.10551.
			
\bibitem{kmptjogo} {\sc P. D. Khanh, B. S. Mordukhovich, V. T. Phat and D. B. Tran}, {\em Generalized damped Newton algorithms in nonsmooth optimization via second-order subdifferentials}, J. Global Optim., to appear (2022), arXiv:2101.10555.
			
\bibitem{kmptmp} {\sc P. D. Khanh, B. S. Mordukhovich, V. T. Phat and D. B. Tran}, {\em Globally convergent coderivative-based generalized Newton methods in nonsmooth optimization}, Math. Program. (second review round), arXiv:2109.02093.
			
\bibitem{kp18} {\sc P. D. Khanh and V. T. Phat}, {\em Second-order characterizations of $\mathcal{C}^1$-smooth robustly quasiconvex functions}, Oper. Res. Lett., 46 (2018), pp.\ 568--572.
			
\bibitem{kp20} {\sc P. D. Khanh and V. T. Phat}, {\em Second-order characterizations of quasiconvexity and pseudoconvexity for differentiable functions with Lipschitzian derivatives}, Optim. Lett., 14 (2020), pp.\ 2413--2427. 

\bibitem{lpr} {\sc A. B. Levy, R. A. Poliquin and R. T. Rockafellar}, {\em Stability of locally optimal solutions}, SIAM J. Optim., 10 (2000), pp.\ 580--604.

\bibitem{mms} {\sc A. Mohammadi, B. S. Mordukhovich and M. E. Sarabi}, {\em Parabolic regularity in geometric variational analysis}, Trans. Amer. Math. Soc., 374 (2021), 1711--1763.

\bibitem{ash-ebr} {\sc A. Mohammadi and M. E. Sarabi}, {\em Twice epi-differentiability of extended-real-valued functions  with applications in composite optimization}, SIAM J. Optim., 30 (2020), pp.\ 2379--2409.
			
\bibitem{m92} {\sc B. S. Mordukhovich}, {\em Sensitivity analysis 
in nonsmooth optimization}, in Theoretical Aspects of Industrial Design, D. A.  Field and V. Komkov, eds., SIAM Proc. Appl. Math., vol.\ 58, Philadelphia, PA, 1992, pp.\ 32--46.
			
\bibitem{Mordukhovich06} {\sc B. S. Mordukhovich}, {\em Variational Analysis and Generalized Differentiation, I: Basic Theory}, Springer, Berlin, 2006.

\bibitem{Mor18} {\sc B. S. Mordukhovich}, {\em Variational Analysis and Applications}, Springer, Cham, Switzerland, 2018.
			
\bibitem{mornam} {\sc B. S. Mordukhovich and N. M. Nam}, {\em Convex Analysis and Beyond, I: Basic Theory}, Springer, Cham, Switzerland, 2022.

 

\bibitem{MorduNghia13} {\sc B. S. Mordukhovich and T. T. A. Nghia}, {\em Second-order variational analysis and characterizations of tilt-stable optimal solutions in infinite-dimensional spaces}, Nonlinear Anal., 86 (2013), pp.\ 159--180.

\bibitem{mnghia-fs} {\sc B. S. Mordukhovich and T. T. A. Nghia}, {\em Full Lipschitzian and H\"olderian stability in optimization with applications to mathematical programming and optimal control}, SIAM J. Optim., 24 (2014), pp.\ 1344--1381.

\bibitem{MorduNghia} {\sc B. S. Mordukhovich and T. T. A. Nghia}, {\em Second-order characterizations of tilt stability with applications to  nonlinear programming}, Math. Program., 149 (2015), pp.\ 83--104.
			
\bibitem{MorduNghia1} {\sc B. S. Mordukhovich and T. T. A. Nghia}, {\em Local monotonicity and full stability of parametric variational systems}, SIAM J. Optim., 26 (2016), pp.\ 1032--1059.

\bibitem{mnr} {\sc B. S. Mordukhovich, T. T. A. Nghia and R. T. Rockafellar}, {\em  Full stability in finite-dimensional optimization}, Math. Oper. Res.,  40 (2016), pp.\ 226--252.	
			
\bibitem{BorisOutrata} {\sc B. S. Mordukhovich and J. V. Outrata}, {\em  On second-order subdifferentials and their applications}, SIAM J. Optim., 12 (2001), pp.\ 139--169.

\bibitem{mo} {\sc B. S. Mordukhovich and J. V. Outrata}, {\em Coderivative analysis of quasi-variational inequalities with applications to stability and optimization}, SIAM J. Optim., 18 (2007), 389--412.

\bibitem{mor} {\sc B. S. Mordukhovich, J. V. Outrata and H. Ram\'irez C.}, {\em Second-order variational analysis in conic programming with applications to optimality and stability}, SIAM J. Optim., 25 (2015), 76--101.

\bibitem{mos} {\sc B. S. Mordukhovich, J. V. Outrata and M. E. Sarabi}, {\em Full stability of locally optimal solutions in second-order cone programming}, SIAM J.
Optim., 24 (2014), 1581--1613.
			
\bibitem{mr} {\sc B. S. Mordukhovich and R. T. Rockafellar}, {\em Second-order
subdifferential calculus with applications to tilt stability in
optimization},  SIAM J. Optim., 22 (2012), pp.\ 953--986.

\bibitem{msar} {\sc B. S. Mordukhovich and M. E. Sarabi}, {\em  Generalized differentiation of piecewise linear functions in second-order variational analysis}, Nonlinear Anal., 132 (2016), 270--273 .

\bibitem{BorisEbrahim} {\sc B. S. Mordukhovich and M. E. Sarabi}, {\em  Generalized Newton algorithms for tilt-stable minimizers in nonsmooth optimization}, SIAM J. Optim., 31 (2021), pp.\ 1184--1214.
		    
\bibitem{os} {\sc J. V. Outrata and D. Sun}, {\em On the coderivative of the projection operator onto the second-order cone},  Set-Valued Anal.,  16 (2008), pp.\ 999--1014.

\bibitem{Poli} {\sc R. A. Poliquin and R. T. Rockafellar}, {\em Tilt stability of a local minimum}, SIAM J. Optim., 8 (1998), pp.\ 287--299.

\bibitem{qw} {\sc N. T. Qui and D. Wachsmuth}, {\em Full stability for a class of control problems of semilinear elliptic partial differential equations}, SIAM J. Control Optim., 57 (2019), pp. 3021--3045.

\bibitem{Rockafellar70} {\sc R. T. Rockafellar}, {\em Convex Analysis}, Princeton University Press, Princeton, NJ, 1970.

\bibitem{r19} {\sc R. T. Rockafellar}, {\em Variational convexity {and the local monotonicity} of subgradient mappings},  Vietnam J. Math., 47 (2019), 547--561.

\bibitem{roc} {\sc R. T. Rockafellar}, {\em Augmented Lagrangians and hidden convexity in sufficient conditions for local optimality}, Math. Program., 192 (2022), DOI 10.1007/s10107-022-01768-w.

\bibitem{r22} {\sc R. T. Rockafellar}, Convergence of augmented Lagrangian methods in extensions beyond nonlinear programming, Math. Program. (2022), DOI 10.1007/s10107-022-01832-5. 
			
\bibitem{Rockafellar98} {\sc R. T. Rockafellar and R. J-B. Wets}, {\em  Variational Analysis}, Springer, Berlin, 1998.
		 
\bibitem{yy} {\sc J.-C. Yao and N. D. Yen}, {\em Coderivative calculation related to a parametric affine variational inequality. Part~1: Basic calculation}, Acta Math. Vietnam., 34 (2009), pp.\ 157--172.
\end{thebibliography}
\end{document}